\documentclass[11pt]{amsart}
\usepackage{epsfig}

\usepackage{graphicx}
\usepackage{amscd}
\usepackage{amsmath}
\usepackage{amsfonts}
\usepackage{mathrsfs}
\usepackage{amssymb}
\usepackage{verbatim}
\usepackage{color}
\usepackage{hyperref}

\newtheorem{theorem}{Theorem}[section]
\theoremstyle{plain}

\newtheorem{corollary}[theorem]{Corollary}

\newtheorem{defi}[theorem]{Definition}
\newtheorem{example}[theorem]{Example}

\newtheorem{lemma}[theorem]{Lemma}
\newtheorem{notation}[theorem]{Notation}

\newtheorem{remark}[theorem]{Remark}

\numberwithin{equation}{section}

\newcommand \modk {{\mathcal M}_{\kappa}}

\def\what{\widehat}

\def\Jk{{\mathcal J}}
\def\Xxi{{\mathfrak X}_\zeta}
\def\Xx{{\mathfrak X}}
\def\FrA{{\mathfrak A}}
\def\FrR{{\mathfrak R}}

\def\Aa{{\mathbb A}}

\def\Ga{{\mathbb G}}
\def\One{{1\!\!1}}

\def\Lip{{\rm Lip}}

\def\Hk{{\mathcal H}}

\def\Zk{{\mathcal Z}}

\def\half{\frac{1}{2}}

\newcommand{\lam}{\lambda}

\newcommand{\gam}{\gamma}
\newcommand{\om}{\omega}

\def\Om{\Omega}
\newcommand{\Sig}{\Sigma}

\newcommand{\Gam}{\Gamma}
\newcommand{\sig}{\sigma}

\newcommand{\R}{{\mathbb R}}

\newcommand{\Z}{{\mathbb Z}}
\newcommand{\C}{{\mathbb C}}

\def\N{{\mathbb N}}

\newcommand{\Prob}{{\mathbb P}\,}
\def\P{\Prob}

\def\bq{{\bf q}}
\def\ba{{\bf a}}
\def\wt{\widetilde}
\def\A{{\mathcal A}}

\def\Rk{{\mathcal R}}

\def\Lk{{\mathcal L}}

\def\Nk{{\mathcal N}}

\def\Pk{{\mathcal P}}

\def\Sf{{\sf S}}
\def\T{{\mathbb T}}
\def\Vk{{\mathcal V}}

\def\be{\begin{equation}}
\def\ee{\end{equation}}

\newcommand{\eps}{{\varepsilon}}

\def\und{\underline}

\def\card{{\rm card}}

\def\Cc{{\mathscr M}}

\def\ve1{\vec{1}}

\def\M{{\mathbf M}}

\def\Ak{{\mathcal A}}

\def\fu{\varphi}

\newcommand{\omb}{\boldsymbol \omega}

\begin{document}

% Sets a PDF bookmark for the dedication
\pdfbookmark{Dedication}{dedication}
\thispagestyle{empty}
\begin{flushright}
  \emph{Dedicated to Larry Zalcman, \\ with admiration and gratitude.}
\end{flushright}

\bigskip

\title{A spectral cocycle for  substitution systems and translation flows}

\author{Alexander I. Bufetov}
\address{Alexander I. Bufetov\\ 
Aix-Marseille Universit{\'e}, CNRS, Centrale Marseille, I2M, UMR 7373\\
39 rue F. Joliot Curie Marseille France }
\address{Steklov  Mathematical Institute of RAS, Moscow}
\address{Institute for Information Transmission Problems, Moscow}
\email{alexander.bufetov@univ-amu.fr, bufetov@mi.ras.ru}
\author{Boris Solomyak }
\address{Boris Solomyak\\ Department of Mathematics,
Bar-Ilan University, Ramat-Gan, Israel}
\email{bsolom3@gmail.com}

\begin{abstract} 
For substitution systems and translation flows, a new cocycle, which we call {\em spectral cocycle}, is introduced, whose Lyapunov exponents govern the local  dimension of the spectral measure for higher-level cylindrical functions. The construction relies on the symbolic representation of translation flows and the formalism of matrix Riesz products. 
 \end{abstract}

\date{\today}


\maketitle

\thispagestyle{empty}

\section{Introduction}

This paper is devoted to the spectral theory of substitution systems and translation flows and continues the work started in \cite{BuSo1,BuSo2}. We focus on the local properties of spectral measures, such as local dimension, H\"older property, and closely related questions of singularity and absolute continuity. Our main construction is that of a new cocycle, which we call the {\em spectral cocycle}. This spectral cocycle is related to  our earlier work, in particular, to the matrix Riesz products, used in  \cite{BuSo1}  to obtain H\"older continuity of spectral measures for typical suspension flows over non-Pisot substitution systems. %translation flows in the stratum $\Hk(2)$. 

Our cocycle is defined over a skew product whose base is a shift transformation on a symbolic space arising from the realization of the Teichm\"uller flow and the fibre is a torus of dimension equal
to the number of intervals in the associated interval exchange. Our main result, see \eqref{d-lyap}, \eqref{d-lyap-gen}  below,  is a formula relating the pointwise dimension of the spectral measure and the pointwise Lyapunov exponent of the cocycle.
As a corollary, we obtain an inequality for the Lyapunov exponent at almost every point and a sufficient condition for singularity of the spectrum. Analogous results are obtained for suspension flows
over $S$-adic systems, including classical substitution systems as a special case.

In the Appendix, which is independent of the rest of the paper, we explain how a modification of the argument from \cite{BuSo2} proves H\"older continuity of spectral measures for typical translation flows in the stratum $\Hk(1,1)$. %thus completing the genus 2 case.

%%%%%%%%%%%%%%%%%%%%%%%%%%%%%%%%%%%%%%%%%%%%

\section{Background}

\subsection{Translation flows}
Let $M$ be a compact connected orientable surface.  To
  a holomorphic one-form $\omb$ on $M$ one can assign  the corresponding {\it vertical} flow $h_t^+$ on $M$, i.e., the flow at unit speed along the
 leaves of the foliation $\Re(\omb)=0$.
The vertical flow preserves the measure 
$
{\mathfrak m}=i(\omb\wedge {\overline \omb})/2, 
$
the area form induced by $\omb$. 
Let $\kappa=(\kappa_1, \dots, \kappa_{\sigma})$ be
a nonnegative integer vector such that  $\kappa_1+\dots+\kappa_{\sigma}=2\rho-2$.
Consider the moduli space $\modk$ of pairs $(M, \omb)$, where $M$ is
a Riemann surface of genus $\rho$ and $\omb$ is a holomorphic differential
of area $1$ with singularities of orders $\kappa_1, \dots, \kappa_{\sigma}$.
The moduli space  $\modk$ is  called the {\it stratum} in the moduli space of abelian differentials.

The Teichm{\"u}ller flow ${\bf g}_s$  sends the
modulus of a pair $(M, \omb)$ to the modulus of the pair
$(M, \omb^{\prime})$, where $\omb^{\prime}=e^s\Re(\omb)+ie^{-s}\Im(\omb)$;
the new complex structure on $M$ is uniquely determined by the requirement that the form $\omb^{\prime}$
 be holomorphic.  Veech \cite{veech3} proved that the space $\modk$ need not be connected;
let $\Hk$ be a connected component of $\modk$, and let $\nu$ be a probability measure on $\Hk$, invariant and ergodic under the Teichm{\"u}ller flow  ${\bf g}_s$. 
For $\nu$-almost every Abelian differential $(M, \omb)$,  Masur \cite{masur} and Veech \cite{veech}  independently proved that the flow $h_t^+$ is uniquely ergodic.  

There are many equivalent definitions of translation surfaces, see, e.g., \cite{MasurTabach,Zorich}. The form $\omb$ provides a flat metric on $M\setminus\Sig$, where $\Sig$ is the set of singularities, the zeros of $\omb$. Moreover, at the zero of order $\kappa_j$ one gets a cone singularity, with the total angle $2\pi(\kappa_j+1)$.

\subsection{Interval exchange transformations and suspensions over them} \label{subsec-IET}

There is  a deep connection between translation flows and {\em interval exchange transformations}, discovered by Veech \cite{Veech0,
veech}. We recall it briefly; for more details see, e.g., the surveys by Viana \cite{viana2}, Yoccoz \cite{yoccoz}, and Zorich \cite{Zorich}. Let $\Ak=\{1,\ldots,m\}$ be a finite alphabet, with $m\ge 2$, and $\pi$  an irreducible permutation of $\Ak$, i.e., $\pi\{1,\ldots,k\}\ne \{1,\ldots,k\}$ for $k<m$. Given a positive vector $\lam\in \R^m_+$, the interval exchange transformation (IET) $f(\lam,\pi)$ is defined as follows: %(see \cite{CFS,Keane}): 
consider the interval $I=[0,\sum_{i=1}^m \lam_i)$, break it into subintervals
$$
I_i = I_i(\lam,\pi) = \Bigl[\sum_{j<i}\lam_j, \sum_{j\le i} \lam_j\Bigr),\ \ 1\le j \le m,
$$
and rearrange the intervals $I_i$ by translation according to $\pi$:
$$
x\mapsto x + \sum_{\pi(j) < \pi(i)} \lam_j - \sum_{j < i} \lam_j,\ \ \ x\in I_i.
$$
For $m=2$ the IET is just a circle rotation (modulo identification of the endpoints of $I$), and it can be viewed as the first return map of a linear flow on a torus $\T^2$. Similarly, for $m\ge 3$ by a singular suspension (with a piecewise-constant roof function, constant on each subinterval $I_i$), the IET can be represented as a first return map of a translation flow on a suitable translation surface to a carefully chosen Poincar\'e section, a line segment $I$, see \cite{Veech0,veech}. Conversely, 
given a translation surface, one can find a horizontal segment $I$ in such a way that the first return map of the vertical flow to $I$ is an IET. Precise connection between the two systems is given by the {\em zippered rectangles construction} of Veech \cite{veech}.

\subsection{Rauzy-Veech-Zorich induction and the corresponding cocycles} A fundamental tool in the study of IET's and translation flows is the Rauzy-Veech algorithm, also called Rauzy-Veech induction, introduced   in \cite{Veech0,Rauzy}. % We describe it briefly, following \cite[4.1]{AF}. 
Let $\pi$ be an irreducible permutation, and suppose that $(\lam,\pi)$ is such that $\lam_m \ne \lam_{\pi^{-1}(m)}$. Then the first return map (what is sometimes called ``inducing'' whence the term ``induction'') of $f(\lam,\pi)$ to the interval
$$
\Bigl(0, \sum_{i=1}^m \lam_i - \min\{\lam_{\pi^{-1}(m)},\lam_m\} \Bigr)
$$
is an irreducible IET on $m$ intervals as well, see, e.g., \cite{viana2,yoccoz,MMY}. If $\lam_m < \lam_{\pi^{-1}(m)}$, we say that this is an operation of type ``a''; otherwise, an operation of type ``b''. The {\em Rauzy graph} is a directed labeled graph, whose vertices are permutations of $\Ak=\{1,\ldots,m\}$ and the edges lead to permutations obtained by applying one of the operations. The edges are labeled ``a'' or ``b'' depending on the type of the operation. The {\em Rauzy class} of a permutation $\pi$ is the set of all permutations that can be reached from $\pi$ following a path in the Rauzy graph. This defines an equivalence relation on the full permutation group. For almost every IET (with respect to the Lebesgue measure on $\R^m_+$, that is, for almost all length vectors), the algorithm is well-defined for all times into the future, that is, we never get into a ``draw'' $\lam_m = \lam_{\pi^{-1}(m)}$, %while the length of the interval on which the interval exchange is performed tends to zero,
and obtain an infinite path in the Rauzy graph, corresponding to the IET. Veech \cite{veech} proved that, conversely, every infinite path in the Rauzy graph arises from an IET in a such a way.

In the ergodic theory of IET's it is useful to consider an {\em acceleration} of the algorithm. Zorich induction \cite{Zorich1} is obtained by applying the Rauzy-Veech induction  until the first switch from a type ``a'' to a type ``b'' operation, or vice versa.
Sometimes other versions of the algorithm and accelerations are used, e.g., the one considered by Marmi, Moussa, and Yoccoz \cite{MMY}.

Let $(\lam',\pi')$ be obtained from $(\lam,\pi)$ by a step of the Rauzy-Veech induction, and let $f_I=f(\lam,\pi)$. Write $I_j = I_j(\lam,\pi)$ and let $J_j = I_j(\lam',\pi')$ be the intervals of the exchange $f_J=f(\lam',\pi')$. Denote by 
$r_i$ the return time for the interval $J_i$ into $J$ under $f_I$, that is,
$r_i = \min\{k>0: f_I^k (J_i) \subset J\}$. From the definition of the induction procedure it follows that $r_i = 1$ for all $i$ except one, for which it is equal to 2. 
Represent $I$ as a Rokhlin tower over the subinterval $J$ and its induced map $f_J$, and write
\be \label{Rokhlin1}
I = \bigsqcup_{i=1,\ldots,m,\ k = 0, \ldots, r_i-1} f_I^k (J_i).
\ee
By construction, each of the ``floors'' of our tower, that is, each of the subintervals $f_I^k (J_i)$ is a subset of a unique subinterval of the initial exchange, and we define an integer $n(i,k)$ by the formula
\be \label{Rokhlin2}
f_I^j (J_i) \subset I_{n(i,k)}.
\ee
Let $B^R(\lam,\pi)$ be the linear operator on $\R^m$ given by the $m\times m$ matrix $[n(i,j)]$. This matrix is easily shown to be unimodular. Given a Rauzy class $\FrR$, the function $B^R: \R^m_+\times \FrR \to GL(m,\R)$ yields the {\em Rauzy-Veech}, or {\em renormalization} cocycle. If, instead, we apply the Zorich induction algorithm, the same procedure yields the {\em Zorich cocycle}. 

One can consider the Rauzy-Veech and Zorich  induction algorithm also on the set of zippered rectangles; these can be represented as bi-infinite paths in the Rauzy graph. A remarkable fact is that, after an appropriate renormalization, the Rauzy-Veech map $(\lam,\pi)\mapsto (\lam',\pi')$ and the Zorich map $(\lam,\pi) \mapsto (\lam'',\pi'')$ can be seen as the first return maps of the Teichm\"uller flow on the space of zippered rectangles, with respect to carefully chosen Poincar\'e sections, see, e.g., \cite[Section 2.10]{viana2} and \cite[Section 11.3]{yoccoz}.

\medskip

\indent
{\bf Remark.} The zippered rectangles construction provides natural bases for the absolute and relative homology groups $H_1(M\setminus \Sig,\R)$ and $H_1(M,\Sig,\R)$; in particular, $\R^m$ may be identified with  $H_1(M,\Sig,\R)$. The Rauzy-Veech cocycle can then be represented as acting on the cohomology groups 
$H^1(M\setminus \Sig,\R)$ and $H^1(M,\Sig,\R)$, as shown by Veech \cite{veech} (see also \cite[Section 2.9]{viana2}).

%%%%%%%%%%%%%%%%%%%%%%%%%%%

\subsection{Markov compacta and $S$-adic systems} \label{sec-Sadic}
The Rauzy-Veech and Zorich induction and cocycles already provide a powerful symbolic framework for the study of IET's and translation flows, which was used by many authors.
However, to get a symbolic representation, or measurable conjugacy for these systems, an additional step is needed, which was done by Bufetov \cite{Buf-umn}, using the theory of {\em Markov compacta}. For the background on Markov compacta and Bratteli-Vershik transformations, see the original papers \cite{Vershik1,Vershik2,Vershik-Livshits}. A Markov compactum is the space of infinite paths in a Bratteli diagram. When the Bratteli diagram is equipped with a {\em Vershik ordering}, one gets a ``transverse,'' or ``adic'' map, which is now usually called the 
Bratteli-Vershik transformation. We will call a Bratteli diagram with a Vershik ordering a {\em Bratteli-Vershik diagram}.

%The theory of 2-sided Markov compacta and their applications to translation flows was developed in \cite{Buf-umn}. 
More recently an  essentially equivalent framework of $S$-adic transformations was developed, see
\cite{BD,BST_2019,BSTY} and references therein. We use it in \cite{BuSo2} and in this paper as well; therefore, we do not present the background on Markov compacta here.
%We only mention that there is a commonly used {\em telescoping operation} on ordered Bratteli diagrams, which results in an almost conjugate Vershik map.

Let $\Ak=\{1,\ldots,m\}$ be a finite alphabet; denote by $\Ak^+$ the set of finite (non-empty) words in $\Ak$.  A {\em substitution} is a map $\zeta:\, \Ak \to \Ak^+$, which is extended to an action on $\Ak^+$ and $\Ak^\N$ by concatenation. The {\em substitution matrix} is defined by 
\be \label{sub-mat}
\Sf_\zeta (i,j) = \mbox{number of symbols}\ i\ \mbox{in the word}\ \zeta(j).
\ee

\begin{notation} \label{not-1}
Denote by $\FrA$ a set of substitutions $\zeta$ on $\Ak$ with the property that all letters appear in the set of words $\{\zeta(a):\,a\in \Ak\}$ and there exists $a\in \Ak$ such that $|\zeta(a)|>1$.
\end{notation}

Let $\ba =(\zeta_n)_{n\ge 1}$ be a sequence of substitutions on $\Ak$. Substitutions, extended to $\Ak^+$, can be composed in the usual way as transformations $\Ak^+\to \Ak^+$. Denote
$$
\zeta^{[n]} := \zeta_1\circ \cdots\circ\zeta_n,\ \ n\ge 1.
$$
Given a sequence of substitutions $\ba$, denote by $X_\ba\subset \Ak^\Z$ the subspace of all two-sided sequences whose every subword appears as a subword of 
$\zeta^{[n]}(b)$ for some $b\in \Ak$ and $n\ge 1$. Let $T$ be the left shift on $\Ak^\Z$; then $(X_{\ba},T)$ is the (topological) $S$-adic dynamical system.
We refer to \cite{BD,BST_2019,BSTY} for the background on $S$-adic shifts. A sequence of substitutions is called {\em primitive} if for any $n\in \N$ there exists $k\in\N$ such that
$\Sf_n\cdots\Sf_{n+k}$ is a matrix with strictly positive entries. This implies minimality of the $S$-adic shift, see \cite[Theorem 5.2]{BD}. (Note that in \cite{BD} this property is called {\em weak primitivity}, however, in \cite{BST_2019,BSTY} the term ``primitive'' is used, as we do.) In many cases a stronger property holds, and it will be one of our basic assumptions:

\medskip

{\bf (A1)} {\em
 There exists a finite word $\zeta_W=\zeta_{w_1}\ldots \zeta_{w_k}$ in the alphabet $\FrA$ which appears in the sequence $\ba$ infinitely often, for which the substitution matrix $\Sf_{\zeta_W}$ is strictly positive.}
 
 \medskip
 
Property (A1), of course, implies primitivity. It also implies unique ergodicity of the $S$-adic shift, see \cite[Theorem 5.7]{BD}; in fact, this goes back to Furstenberg \cite[(16.13)]{Furst}.

We will also assume that the $S$-adic system is {\em aperiodic}, i.e., it has no periodic points. (A minimal system that has a periodic point, is obviously a system on a finite space, and we want to exclude a trivial situation.) Checking aperiodicity may require some work, even for a single substitution.

Further, we need the notion of  {\em recognizability} for the sequence of substitutions, introduced %by Berth\'e, Steiner, Thuswaldner, and Yassawi 
in \cite{BSTY}, which generalizes {\em bilateral recognizability} of B. Moss\'e \cite{Mosse} for a single substitution, see also Sections 5.5 and 5.6 in \cite{Queff}. By definition of the space $X_\ba$, for every $n\ge 1$, every $x\in X_{\ba}$ has a representation of the form
\be \label{recog}
x = T^k\bigl(\zeta^{[n]}(x')\bigr),\ \ \mbox{where}\ \ x'\in X_{\sig^n \ba},\ \ 0\le k < |\zeta^{[n]}(x_0)|.
\ee
Here $\sig$ denotes the left shift on $\FrA^\N$, and we recall that a substitution $\zeta$ acts on $\Ak^\Z$ by
$$
\zeta(\ldots a_{-1}.a_0 a_1\ldots) = \ldots \zeta(a_{-1}).\zeta(a_0)\zeta(a_1)\ldots
$$

\begin{defi} \label{def-recog}
A sequence of substitutions $\ba = (\zeta_j)_{j\ge 1}$ is said to be {\em recognizable} if the representation (\ref{recog}) is unique for all $n\ge 1$.
\end{defi}

The following is a special case of \cite[Theorem 4.6]{BSTY} that we need.

\begin{theorem}[{\cite{BSTY}}] \label{th-recog0}
Let $\ba= (\zeta_j)_{j\ge 1} \in \FrA^\N$ be such that every substitution matrix $\Sf_{\zeta_j}$ has maximal rank $m$ and $X_\ba$ is aperiodic. Then $\ba$ is recognizable.
\end{theorem}

There is a canonical correspondence between (one-sided) Bratteli-Vershik diagrams with $m$ vertices on each level  and sequences of substitutions $\ba=(\zeta_j)_{j\ge 1}$ on the alphabet 
$\Ak = \{0,\ldots,m-1\}$, discovered by Livshits \cite{Liv1}. 

\begin{theorem}[{\cite[Theorem 6.5]{BSTY}}] \label{thm-recog}
Let $\sig\in \FrA^\N$ be a recognizable sequence of substitutions. Then the $S$-adic shift $(X_\ba,T)$ is almost topologically conjugate, hence measurably conjugate in case the system is uniquely ergodic, to the corresponding Bratteli-Vershik system.
\end{theorem}

%%%%%%%%%%%%%%%%%%%%%%%%%%%%%%%%%%

\subsection{Symbolic representation of IET's and translation flows} \label{subsec-symbol}

Let $\Hk$ be a connected component of a stratum and $\FrR$ the Rauzy class of a permutation corresponding to $\Hk$.
Veech \cite{veech}
constructed a measurable  map from the space $\Vk(\FrR)$ of
zippered rectangles corresponding to the Rauzy class $\FrR$, to $\Hk$, which intertwines the Teichm\"uller flow on $\Hk$ and a renormalization flow $P_t$ that Veech defined on $\Vk(\FrR)$.
%Our convention here follows that of \cite{Buf-umn}. 
Section 4.3 of \cite{Buf-umn}  gives a symbolic coding of the flow $P_t$ on $\Vk(\FrR)$ on a space of 2-sided Markov compacta with a Vershik ordering. Using the canonical correspondence with sequences of substitutions, we obtain a map
\be \label{map-veech}
\Zk_{\FrR}: (\Vk(\FrR),\wt\nu) \to (\Om,\P)
\ee
to a probability space of 2-sided sequences of substitutions $(\zeta_j)_{j\in \Z}\in \FrA^\Z$,
defined almost everywhere. Here $\wt\nu$ is the pull-back of  $\nu$, an invariant and ergodic map under the Teichm\"uller flow on
$\Hk$. The first return map of the flow $P_t$ for an appropriate Poincar\'e section is mapped by $\Zk_\FrR$ to the shift map $\sigma$ on $(\Om,\P)$. This correspondence maps the Rauzy-Veech cocycle over the Teichm\"uller flow into the renormalization cocycle associated with the sequence of substitutions. More precisely, the substitution $\zeta_1$ in the symbolic representation of \cite{Buf-umn} can be ``read off'' the Rokhlin tower (\ref{Rokhlin1}), (\ref{Rokhlin2}) of one step of the Rauzy-Veech induction:
\be \label{induc}
\zeta_1:\ i \mapsto n(i,0)\ldots n(i, r_1-1),\ \ i=1,\ldots,m.
\ee
Thus we obtain
$$B^R(\lam,\pi) = [n(i,j)]_{i,j=1}^m = \Sf_{\zeta_1}^t.$$ We will be using the following notation for this cocycle:
 \be \label{Rauzy-Veech1}
 \Aa(\ba):= \Sf_{\zeta_1}^t;\ \ \Aa(\ba,n) := \Aa(\sig^{n-1}\ba)\cdot \ldots \cdot \Aa(\ba),
 \ee
 where $\ba = (\zeta_j)_{j=1}^\infty$ is the positive side of a sequence of substitutions from $\Om$.

A zippered rectangle $\Rk\in \Vk(\FrR)$ determines a suspension flow over an IET, isomorphic to the translation flow on a flat surface. The symbolic coding 
 $\Zk_\FrR$ induces a map defined for a.e.\ $\Rk\in \Vk(\FrR)$, from the corresponding flat surface $M(\Rk)$ to a suspension over the $S$-adic space $X_\ba$.
Moreover this map takes the IET into the $S$-adic system $(X_\ba,T)$, whereas the piecewise-constant roof function of the suspension is determined by the left side of the sequence $(\zeta_j)_{j=-\infty}^0$. The justification for transition from the Bratteli-Vershik coding of \cite{Buf-umn} to the $S$-adic framework is provided by Theorems~\ref{thm-recog} and \ref{th-recog0}, in view of the fact that the matrices of the Rauzy-Veech cocycle are unimodular, see \cite{veech,veechamj}, hence have maximal rank.
 We denote by $\Om_+$ the projection of $\Om$ to the ``positive side'' and by $\P_+$ the projection of the measure $\P$ to $\Om_+$.
The property (A1) holds
for $\P_+$-almost every $\ba\in \Om_+$, see Veech \cite{veech}.

\subsection{Cylindrical functions}
Suppose that the $S$-adic system $(X_\ba,T)$ is  uniquely ergodic, with the unique invariant probability measure by $\mu$. 
Consider the partition of $X_\ba$  into cylinder sets according to the value of $x_0$: $X_{\ba} = \bigsqcup_{a\in \Ak} [a] $.
Denote
by $(\Xx_\ba^{\vec{s}}, h_t, \wt{\mu})$ the suspension flow over $(X_{\ba},\mu, T)$,  corresponding to a piecewise-constant roof function determined by $\vec{s}\in \R^m_+$. 
We have a union, disjoint in measure:
$$
\Xx_\ba^{\vec{s}} = \bigcup_{a\in \Ak} [a]\times [0,s_a],
$$
and define a {\em Lip-cylindrical function} by the formula: 
\be \label{fcyl2}
f(x,t)=\sum_{a\in \Ak} \One_{[a]}(x) \cdot \psi_a(t),\ \ \mbox{with}\ \ \psi_a\in \Lip[0,s_a],
\ee
where $\Lip$ is the space of Lipschitz functions.

%%%%%%%%%%%%%%%%%%%%%%

\section{Statement of main results}

\subsection{Definition of the spectral cocycle}
We proceed to the main construction of the paper. Let $\zeta$ be a substitution on $\Ak$ with a substitution matrix having non-zero determinant.
Consider  the toral endomorphism $\xi\mapsto \Sf_\zeta^t \,\xi\ (\mbox{mod}\ \Z^m),\ \xi \in \T^m = \R^m/\Z^m$,  induced by the transpose substitution matrix. Suppose that
$$\zeta(b) = u_1^{b}\ldots u_{|\zeta(b)|}^{b},\ \ b\in \Ak.$$
\begin{defi} \label{def-matrix}
Define a matrix-valued function $\Cc_\zeta:  \R^m\to M_m(\C)$  (the space of complex $m\times m$ matrices) by the formula 
\be \label{coc0}
\Cc_\zeta(\xi) = [\Cc_\zeta(\xi_1\ldots,\xi_m)]_{(b,c)} := \Bigl( \sum_{j\le |\zeta(b)|,\ u_j^{b} = c} \exp\bigl(-2\pi i \sum_{k=1}^{j-1} \xi_{u_k^{b}}\bigr)\Bigr)_{(b,c)\in \A^2},\ \ \ \xi\in \R^m.
\ee
Note that  $\Cc_\zeta$ is $\Z^m$-periodic, so we obtain a continuous matrix-function on the torus, which we denote, by a slight abuse of notation, by the same letter:
$\Cc_\zeta: \T^m\to M_m(\C)$.
\end{defi}

\begin{remark} {\em
The matrix $\Cc_\zeta$ already appeared (with a different notation) in \cite[(4.15)]{BuSo1} in the framework of generalized matrix Riesz products. It is also closely related to the Fourier matrix $B(k)$ from the recent papers of Baake et al. \cite{BFGR,BGM,BG2M}. More precisely, $B(k)$ is the restriction of $\Cc_\zeta$ to the line $\{\xi = k\vec s, \ k\in \R\}$, where $\vec s$ is the Perron-Frobenius eigenvector of $\Sf_\zeta^t$. This leads to a cocycle on the line $\R$, used to study the diffraction spectrum of a single substitution, see Section~\ref{subsec-Baake} below for a more detailed discussion.}
\end{remark}

\begin{example}
Let $\Ak = \{1,2,3\}$, 
$$\zeta(1) = 121321,\ \ \zeta(2) = 2231,\ \ \zeta(3) = 31123.$$
Denoting $z_j := e^{-2\pi i \xi_j}$, $j\le 3$, we obtain
$$
\Cc_\zeta(\xi_1,\xi_2,\xi_3) = \Cc_\zeta(z_1,z_2,z_3) = \left(\begin{array}{ccc} 1 + z_1z_2 + z_1^1 z_2^2 z_3 & z_1 + z_1^2 z_2 z_3 & z_1^2 z_2 \\
                                                                                                                            z_2^2 z_3 & 1 + z_2 & z_2^2 \\
                                                                                                                            z_3 + z_1 z_3 & z_1^2 z_3 & 1 + z_1^2 z_2 z_3 \end{array} \right).
$$
\end{example}

Observe that $\Cc_\zeta(\xi)$ is a matrix-function whose entries are trigonometric polynomials in $m$ variables, with the following properties: (i) $\Cc_\zeta(0) = \Sf^t_\zeta$, (ii) all the coefficients are 0's and 1's, (iii) in every row, any given monomial appears at most once, (iv) the maximal degree of the entries in $j$-th row equals $|\zeta(j)|-1$, (v) the substitution is uniquely determined by 
$\Cc_\zeta$. The most important property is
\be \label{most}
\Cc_{\zeta_1\circ \zeta_2}(\xi) = \Cc_{\zeta_2}(\Sf^t_{\zeta_1}\xi)\Cc_{\zeta_1}(\xi),
\ee
which is verified by a  direct computation.

\medskip

Consider the skew product transformation $\Ga:\, \Om_+ \times \T^m \to \Om_+\times \T^m$ defined by
$$
\Ga(\ba, \xi) = \bigl(\sig \ba, \Sf^t_{\zeta_1} \xi \, ({\rm mod}\ \Z^m)\bigr),\ \ \mbox{where}\ \ \ba = (\zeta_n)_{n\ge 1}\ \ \mbox{and}\ \ \xi \in \T^m = \R^m/\Z^m.
$$

\begin{defi} \label{def-cocycle2}
Let $
\Cc(\ba,\xi) = \Cc_{\zeta_1}(\xi)$,
where $\ba = (\zeta_n)_{n\ge 1}$.
Then
\begin{equation}\label{def-cc}
\Cc_{_\Om}((\ba,\xi),n):= \Cc(\Ga^{n-1}(\ba,\xi))\cdot \ldots \cdot \Cc(\ba,\xi)
\end{equation}
is a complex matrix cocycle over the skew product system $(\Om_+\times\T^m, \P\times \nu_m, \Ga)$, where $\nu_m$ is the Haar measure on $\T^m$.
\end{defi}

\noindent 
The following is immediate from definitions and (\ref{most}):

\begin{lemma}
{\rm (i)} The spectral cocycle is an extension of the Rauzy-Veech cocycle in the form (\ref{Rauzy-Veech1}), namely,
 $$\Cc_{_\Om}((\ba,0),n) = \Aa(\ba,n).$$
 
{\rm (ii)} We have
$$
\Cc_{_\Om}((\ba,\xi),n) = \Cc_{\zeta^{[n]}}(\xi).
$$
\end{lemma}

It follows from (ii)  that the spectral cocycle behaves consistently with a ``telescoping'' operation of replacing a finite sequence of substitutions by their composition; in particular, applying the Zorich acceleration algorithm we obtain the spectral cocycle that is an extension of the Zorich cocycle.

Introduce the pointwise upper  Lyapunov exponent of our cocycle, corresponding to a given vector $\vec{z}\in \C^m$: 
\be \label{Lyap0}
{\chi}_{\ba,\xi,\vec{z}}^+:= \limsup_{n\to \infty} \frac{1}{n} \log \|\Cc_{_\Om}((\ba,\xi),n)\vec{z}\|.
\ee
We also consider the pointwise upper Lyapunov exponent (which is independent of the matrix norm):
$$
{\chi}_{\ba,\xi}^+:= \limsup_{n\to \infty} \frac{1}{n} \log \|\Cc_{_\Om}((\ba,\xi),n)\|.
$$
%In Section~\ref{section-ergodic} we will show that the skew product $\Ga$ is ergodic, and then 
%Furstenberg-Kesten Theorem \cite{FK} yields that the top Lyapunov exponent exists and is constant almost everywhere:
%\be \label{FK}
%\exists\, \chi(\Cc_{_\Om}) = \lim_{n\to\infty} \frac{1}{n}\log\|\Cc_{_\Om}((\ba,\xi),n)\|\ \ \ \mbox{for $\Prob_+\times \nu_m$-a.e.}\ (\ba,\xi)\in \Om_+\times\T^m.
%\ee

Let $\lam$ be the Lyapunov exponent of the Rauzy-Veech cocycle, which exists by ergodicity of the Teichm\"uller flow:
\be \label{RVLyap}
\lam = \lim_{n\to\infty} \frac{1}{n} \log\|\Aa(\ba,n)\| \ \ \mbox{for $\Prob_+$-a.e.}\ \ba\in \Om_+.
\ee
Observe that for all $\xi\in \T^m$ the entries of the matrix $\Cc_{_\Om}((\ba,\xi),n)$ are not greater than the corresponding entries of the positive $\Aa(\ba,n)$. It follows that
$$
\chi^+_{\ba,\xi} \le \lam\ \ \mbox{for $\P_+$-a.e.}\ \ba\in \Om_+.
$$
Notice that 
\be \label{Lyapa}
\chi_{\ba,\xi}^+\ge 0 \ \ \mbox{for all}\ \ba\in \Om_+\ \mbox{and}\ \xi\in \T^m,
\ee
because of the special property of substitutions $\zeta_j$ appearing in the Rauzy-Veech induction \cite{veech,veechamj}; one can check with the help of (\ref{induc}) that
$$
|\det \Cc_{\zeta_j}(\xi)|\equiv 1, \xi\in \T^m,\ \ \mbox{for all}\ j\in \N,
$$
and the norm of $m\times m$ matrix is not smaller than the $m$-th root of the absolute value of its determinant.

\subsection{Statement of results} Our main result is a formula expressing the lower local dimension of spectral measures in terms of the Lyapunov exponents of the spectral cocycle. Recall that, given a probability measure-preserving flow $h_t$ on a space $\Xx$ and a test function $f\in L^2(\Xx)$ the spectral measure $\sig_f$ is a finite positive Borel measure on $\R$ defined by
$$
\widehat{\sig}_f(-t) = \int_{-\infty}^\infty e^{2 \pi i\om t}\,d\sig_f(\om) = \langle f\circ h_t, f\rangle,\ \ \ t\in \R,
$$
where $\langle \cdot,\cdot \rangle$ denotes the inner product in $L^2$.

We are interested in fractal properties of spectral measures. One of the commonly used local characteristics of a finite  measure $\nu$ on a metric space is the {\em lower local dimension} defined 
by $$\underline{d}(\nu,\om) = \liminf_{r\to 0} \frac{\log\nu(B_r(\om))}{\log r}\,.$$ Alternatively, it can be thought of as the best possible local H\"older exponent, i.e.:
$$
\und{d}(\nu,\om) = \sup\{\alpha\ge 0:\ \nu(B_r(\om)) = O(r^\alpha),\ r\to 0\}.
$$
For example, the lower local dimension is zero at a point mass and is infinite outside the compact support of a measure.

Next we define the class of test functions under consideration.
Recall that $\Hk$ is a connected component of the moduli space of Abelian differentials on a surface of genus $g\ge 2$, equipped with a probability measure $\nu$, 
invariant  and ergodic under the Teichm\"uller flow ${\bf g}_s$.
We fix a symbolic representation of the Teichm\"uller flow as described in Section~\ref{subsec-symbol}.
 Under this symbolic representation, the vertical flow on the flat surface becomes the suspension flow over an 
 $S$-adic system. So, for $\nu$-a.e.\ flat surface $(M,\omb)$, corresponding to a pair $(\ba,\vec{s})$, with $\ba\in \Om_+$ and $\vec{s}\in \R^m_+$,
 almost every point on $M$ is mapped  into a pair $(x,t)$, where $x\in X_\ba$ (the $S$-adic space) and $t \in [0,s_{x_0}]$. Our test functions will be Lip-cylindrical functions of the form (\ref{fcyl2}) in this symbolic representation. 
    
\begin{theorem} \label{th-main1}
Let $\nu$ be a probability measure on $\Hk$, invariant  and ergodic under the Teichm\"uller flow ${\bf g}_s$.
For $\nu$-almost every  Abelian differential $(M,\omb)$ the following holds. Let $(M, \omb)$ correspond to a pair $(\ba,\vec{s})$, with $\ba\in \Om_+$ and $\vec{s}\in \R^m_+$. Let $f(x,t) = \sum_{a\in \Ak}\One_{[a]}(x)\cdot \psi_a(t)$ be a Lip-cylindrical function and $\sig_f$ the spectral measure of $f$ for the suspension flow $(\Xx_\ba^{\vec{s}}, {\widetilde \mu}, h_t)$, which is measurably isomorphic to the vertical
translation flow on $(M,\omb)$.
 Fix $\om\in \R$ (the spectral parameter), let
$\xi\in \T^m$ be such that 
$$
\xi = \om \vec{s}\ (\mbox{\rm mod}\ \Z^m),
$$
and set
$$
\vec{z} = (\what{\psi}_a(\om))_{a\in \Ak}.
$$
Suppose that the upper Lyapunov exponent of the spectral cocycle at $(\ba,\xi)$, corresponding to the vector $\vec{z}$ is positive: ${\chi}^+_{\ba,\xi,\vec{z}}> 0$.
Then
 \begin{equation}\label{d-lyap}
 \underline{d}(\sig_f,\om) = 2 - \frac{2{\chi}^+_{\ba,\xi,\vec{z}}}{\lam}\,.
 \end{equation}
 If ${\chi}^+_{\ba,\xi,\vec{z}}\le 0$, then
$$
\underline{d}(\sig_f,\om) \ge 2.
$$
\end{theorem}

%It would be interesting to determine when the Lyapunov exponents ${\chi}^+_{\ba,\xi,\vec{z}}$ are positive.

\medskip

Fix $\vec{s}\in \R^m_+$. Theorem~\ref{th-main1} shows that the cocycle $\Cc_{_\Om}(\om\vec{s},n)$ controls the behaviour of spectral measures for the suspension flow $(\Xx_\ba^{\vec{s}}, {\widetilde \mu}, h_t)$.
Here and below we consider $\om\vec{s}$ as a point on the torus (mod $\Z^m)$.

As a special case of Lip-cylindrical functions,  consider {\em simple cylindrical functions}, of the form $f(x,t) = \sum_{a\in \Ak} b_a \One_{[a]}(x)$, that is,
$\psi_a(\om) = b_a\One_{[0,s_a]}$. Then 
\be \label{cylia}
\vec{z} = (\what{\psi}_a(\om))_{a\in \Ak} = \left(\frac{b_a(1-e^{-2\pi i \om s_a})}{2\pi i \om}\right)_{a\in \Ak}=: \Gam_\om \vec{b},\ \ \mbox{for}\ \om\ne 0,
\ee
and for $\om=0$ we have $\vec{z} = \Gam_0(\vec{b})=\vec{b}$.
Let $\vec{e}_j$ be the $j$-th unit coordinate vector in $\C^m$.

\begin{corollary} \label{cor-spec0}  Suppose that we are under the assumptions of Theorem~\ref{th-main1}, and let $\vec{s}\in \R_+^m$.
 Consider the suspension flow $(\Xx_\ba^{\vec{s}}, {\widetilde \mu}, h_t)$. 

{\rm (i)} Let $\om\in \R\setminus \bigcup_{a\in\Ak} (s_a^{-1}\cdot \Z)$.  For any $\vec{b}\in \C^m$ there exists $j\le m$ such that
$$
\card\Bigl\{c\in \C:\ \chi^+_{\ba,\om\vec{s},\, \Gam_\om(\vec{b} + c\vec{e}_j)}< \chi^+_{\ba,\om\vec{s}}\Bigr\} \le 1,
$$
and hence for all perturbations of $\vec{b}$ along the $j$-th direction, except possibly one, the simple cylindrical function $f$ corresponding to it, satisfies
$$
\und{d}(\sig_f,\om) = 2 - \frac{2\chi^+_{\ba,\om\vec{s}}}{\lam}\,, \  \mbox{assuming}\  \chi^+_{\ba,\om\vec{s}}>0.
$$

{\rm (ii)} For Lebesgue-a.e.\ $\vec{b}\in \C^m$, the simple cylindrical function $f$ corresponding to $\vec{b}$ satisfies
$$
\und{d}(\sig_f,\om) = 2 - \frac{2\chi^+_{\ba,\om\vec{s}}}{\lam}\ \ \mbox{for Lebesgue-a.e.}\ \om\in \R, \ \mbox{such that}\  \chi^+_{\ba,\om\vec{s}}>0.
$$
\end{corollary}

Recall that $\chi_{\ba,\om\vec{s}}^+\ge 0$ by (\ref{Lyapa}). It would be interesting to determine under what conditions the Lyapunov exponents $\chi_{\ba,\om\vec{s}}^+$ are strictly positive.

\begin{corollary} \label{cor-spec1} Suppose that we are under the assumptions of Theorem~\ref{th-main1}, and let $\vec{s}\in \R_+^m$.
Then for a.e.\ $\ba$ the following hold.

{\rm (i)} we have 
\be\label{strange}
\chi_{\ba,\om\vec{s}}^+ 
\le  \half\lam\ \ \mbox{ for Lebesgue-a.e.\ $\om\in \R$},
\ee
but
\be\label{strange2}
\chi_{\ba,\om\vec{s}}^+ 
\ge  \half\lam\ \ \mbox{ for $\sig_f$-a.e.\ $\om\in \R$}.
\ee

{\rm (ii)}
If the flow  $(\Xx_\ba^{\vec{s}}, {\widetilde \mu}, h_t)$ has a non-trivial absolutely continuous component of the spectrum, then
$\chi_{\ba,\om\vec{s}}^+ =  \half\lam_\ba$ for Lebesgue-a.e.\ $\om\in \R$.
\end{corollary}

\begin{remark} {\em
In the paper \cite{BuSo2} we proved H\"older continuity of spectral measures for almost every translation flow in the stratum $\Hk(2)$, essentially via showing that there exists $\eps>0$ such that
\be \label{unif}
\chi_{\ba,\om\vec{s}}^+ \le \lam-\eps\ \ \mbox{for {all} $\om\ne 0$, for a.e.}\ \vec{s}.
\ee
Although this is much weaker than the estimate in (\ref{strange}) for a single $\om$, the difference is that (\ref{unif}) is obtained for {\em all} non-zero spectral parameters, rather than for Lebesgue-almost all.}
\end{remark}

%%%%%%%%%%%%%%%%%%%%%%%%%%%

\subsection{Higher-level cylindrical functions} Cylindrical functions used so far in this paper do not suffice to describe the spectral type of the flow completely. Rather, we need functions depending on an arbitrary fixed number of symbols. Fix $\ell\ge 1$. Cylindrical functions of level $\ell$ depend on the first $\ell$ edges of the path representing a point in the Bratteli-Vershik representation. In the 
$S$-adic framework, we say that $f$ is a Lip-cylindrical function of level $\ell$ if
\be \label{fcyl3}
f(x,t)=\sum_{a\in \Ak} \One_{\zeta^{[\ell]}[a]}(x) \cdot \psi^{(\ell)}_a(t),\ \ \mbox{with}\ \ \psi^{(\ell)}_a\in \Lip[0,s^{(\ell)}_a],
\ee
where
$$
\vec{s}^{\,(\ell)}= (s^{(\ell)}_a)_{a\in \Ak}:= \Sf_{\zeta^{[\ell]}}^t \vec{s}.
$$
This representation depends on the notion of recognizability for the sequence of substitutions, see Definition~\ref{def-recog}, as will be explained in the next section.
The following is an extension of Theorem~\ref{th-main1} to the case of higher-order cylindrical functions.

\begin{theorem} \label{th-main11} Let $\nu$ be a probability measure on $\Hk$, invariant  and ergodic under the Teichm\"uller flow ${\bf g}_s$.
For $\nu$-almost every  Abelian differential $(M,\omb)$ the following holds.   Let $(M, \omb)$   correspond to a pair $(\ba,\vec{s})$, with $\ba\in \Om_+$ and $\vec{s}\in \R^m_+$. Let $f(x,t) = \sum_{a\in \Ak} \One_{\zeta^{[\ell]}[a]}(x) \cdot \psi^{(\ell)}_a(t)$ be a Lip-cylindrical function of level $\ell$ and $\sig_f$ the spectral measure of $f$ for the suspension flow $(\Xx_\ba^{\vec{s}}, {\widetilde \mu}, h_t)$. Fix $\om\in \R$, let
$\xi\in \T^m$ be such that 
$$
\xi = \om \vec{s^{(\ell)}}\ (\mbox{\rm mod}\ \Z^m),
$$
and set
$$
\vec{z} = (\what{\psi}^{(\ell)}_a(\om))_{a\in \Ak}.
$$
Suppose that the upper Lyapunov exponent of the spectral cocycle at $(\ba,\xi)$, corresponding to the vector $\vec{z}$ is positive: ${\chi}^+_{\ba,\xi,\vec{z}}> 0$.
Then
 \begin{equation}\label{d-lyap-gen}
 \underline{d}(\sig_f,\om) = 2 - \frac{2{\chi}^+_{\ba,\xi,\vec{z}}}{\lam}\,.
 \end{equation}
 If ${\chi}^+_{\ba,\xi,\vec{z}} \le 0$, then
$$
\underline{d}(\sig_f,\om) \ge 2.
$$
\end{theorem}

This theorem has corollaries that are exact analogues of Corollaries~\ref{cor-spec0} and \ref{cor-spec1} for the higher-level cylindrical functions.

%%%%%%%%%%%%%%%%%%%%%%%%%%%%%%%%%%%%%%%%%%%%%%%%%%%%%

\section{Suspensions over $S$-adic systems}

The results of the previous section are deduced from more general results, obtained for suspensions over an {\em individual} $S$-adic system, satisfying some natural assumptions. As a special case, we obtain results for classical substitution systems.

Suppose we are given a sequence of substitutions $\ba=(\zeta_j)_{j\ge 1}\in \FrA$.
We continue to use the notation $\Sf_\zeta$ for the substitution matrix of $\zeta$, see (\ref{sub-mat}). Observe that $\Sf_{\zeta_1\circ \zeta_2} = \Sf_{\zeta_1}\Sf_{\zeta_2}$. In this section we do not assume that $\det\Sf_\zeta\ne 0$, allowing degenerate substitution matrices.
We will sometimes denote
$$
\Sf_j:= \Sf_{\zeta_j}\ \ \mbox{and}\ \ \Sf^{[n]}:= \Sf_{\zeta^{[n]}}.
$$
We will also consider subwords of the sequence $\ba$ and the corresponding substitutions obtained by composition. We write
$$
\Sf_\bq = \Sf_n\cdots \Sf_\ell\ \ \mbox{for}\ \ \bq = \zeta_{n}\ldots\zeta_{\ell}.
$$
We will be using matrix 1-norm; recall that it equals the maximal absolute column sum of the matrix. Thus, $${\|\Sf_\zeta\|}_1 = \max_{a\in \Ak}|\zeta(a)|,$$
by the definition of the substitution matrix.

Recall that, given a sequence of substitutions $\ba$, we denote by $X_\ba\subset \Ak^\Z$ the subspace of all two-sided sequence whose every subword appears as a subword of 
$\zeta^{[n]}(b)$ for some $b\in \Ak$ and $n\ge 1$. Let $T$ be the left shift on $\Ak^\Z$; then $(X_{\ba},T)$ is the (topological) $S$-adic dynamical system.
We restate our basic standing assumption in the current notation:

\medskip

{\bf (A1)} {\em There  is a word $\bq$ which appears in $\ba$ infinitely often, for which $\Sf_\bq$ has all entries strictly positive.}

\medskip

For random sequences of substitutions, a stronger condition (A1$'$), which provides a certain regularity of appearances of the ``good word'' $\bq$,
holds  by the Ergodic Theorem, see Lemma~\ref{lem-vspoma1} below:

\medskip

{\bf (A1$'$)} {\em For any $\eps>0$ there exists $n_0=n_0(\eps)$ such that for all $n\ge n_0$ the word $\bq$, with $\Sf_\bq$ strictly positive, appears as a subword of 
$\zeta_{\lfloor n(1-\eps) \rfloor+1}, \zeta_{\lfloor n(1-\eps) \rfloor+2},\ldots, \zeta_n$.}

\medskip

As explained in Section~\ref{sec-Sadic}, condition (A1) implies that $(X_\ba,T)$ is minimal and uniquely ergodic.
Let $\mu$ be the unique invariant Borel probability measure for the system $(X_\ba,T)$. Consider $(\Xx_\ba^{\vec{s}}, \wt{\mu}, h_t)$, the suspension flow over $(X_{\ba},\mu, T)$,  corresponding to a piecewise-constant roof function determined by $\vec{s}\in \R^m_+$. The second basic assumption is the existence of the Lyapunov exponent for the non-stationary Markov chain:

\medskip

{\bf (A2)} {\em The following limit exists:}
\be \label{FK0}
\lam_\ba:= \lim_{n\to \infty} \frac{1}{n} \log\|\Sf^{[n]}\|<\infty.
\ee

 \begin{defi} \label{def-cocycle} Given a substitution $\zeta$,
we consider the complex matrix-valued function $\Cc_\zeta(\xi)$ defined on $\R^m$ and on the torus $\T^m$ by formula (\ref{coc0}).
For a sequence of substitutions $\ba = (\zeta_n)_{n\ge 1}$, let
$$
\Cc_\ba(\xi,n):= \Cc_{\zeta_n}\bigl(\Sf^t_{\zeta^{[n-1]}}\xi \bigr)\cdot \ldots \cdot \Cc_{\zeta_2}(\Sf^t_{\zeta_1}\xi)\Cc_{\zeta_1}(\xi),\ \ n\ge 1.
$$
Under the assumption $\det(\Sf_{\zeta_j})\ne 0$, $j\ge 1$, the matrix $\Cc_\ba(\xi,n)$ may be viewed as a complex matrix cocycle on the torus $\T^m$ over the
 non-stationary toral endomorphism $E_{\zeta_n}:\T^m\to \T^m$ at time $n$, where $E_\zeta(\xi) = \Sf^t_\zeta\xi$ {\rm (mod} $\Z^m${\rm )}.
\end{defi}

Similarly to the previous section, we consider the pointwise upper  Lyapunov exponent of our cocycle, corresponding to a given vector $\vec{z}\in \C^m$: 
\be \label{Lyap1}
{\chi}_{\ba,\xi,\vec{z}}^+=\limsup_{n\to \infty} \frac{1}{n} \log \|\Cc_\ba(\xi,n)\vec{z}\|, 
\ee
as well as
$$
{\chi}_{\ba,\xi}^+= \limsup_{n\to \infty} \frac{1}{n} \log \|\Cc_\ba(\xi,n)\|.
$$
Note that $\Cc_\ba(0,n) = \Sf_{\zeta^{[n-1]}}^t$, and for all $\xi$ the absolute values of the entries of $\Cc_\ba(\xi,n)$ are not greater than those of $\Sf_{\zeta^{[n-1]}}^t$. Therefore,
$$\chi_{\ba,\xi}^+ \le \lam_\ba,\ \ \mbox{ for all}\ \xi\in \T^m.$$
where $\lam_\ba$ is defined in (\ref{FK0}).
We need to add one more requirement on the sequence $\ba$.
%In the case of a random $S$-adic systems, arising from some ergodic stationary process, as in the previous section, these properties are immediate from the Birkhoff-Khintchine  Ergodic Theorem.

\medskip

{\bf (A3)} {\em We have}
\begin{equation} \label{ass2}
\lim_{n\to \infty} \frac{1}{n} \log(1+\|\Sf_n\|) = 0.
\end{equation}

In the case of a ``random'' $S$-adic system, as in the previous section, these properties are deduced with the help of the next lemma. Let $\Om_+$ be a shift-invariant subspace of $\FrA^\N$ (see Notation~\ref{not-1}), with an invariant measure $\P_+$. %This includes the case of $S$-adic systems arising from typical translation flows from Section 3, but here the set-up is more general.

\begin{lemma} \label{lem-vspoma1} Suppose that the following properties hold:

{\bf (C1)} the system $(\Om_+,\sig,\P_+)$ is ergodic;

{\bf (C2)} the function $\ba \mapsto \log(1 + \|\Sf_{\zeta_1}\|)$ is integrable;

{\bf (C3)} there is a word $\bq$ admissible for sequences in $\Om_+$, such that all entries of the matrix $\Sf_\bq$ are positive and $\P_+([\bq])>0$.

Then conditions {\bf (A1$'$)}, {\bf (A2)}, and {\bf (A3)} hold $\P_+$-almost surely.
\end{lemma}

For the proof, note that property (A2) for $\P_+$-a.e.\ $\ba$ follows from  (C1) and (C2) by the Furstenberg-Kesten theorem on the existence of Lyapunov exponent. The a.s.\ validity of (A3) is immediate from the Birkhoff-Khinchin Ergodic Theorem and (C2). Finally, the
property (A$1'$) for $\P_+$-a.e.\ $\ba$ is a consequence of (C1), (C3), and the following standard fact.

\medskip

\noindent {\bf Observation.} {\em
Let $(X,T,\mu)$ be an ergodic measure-preserving system and $A\subset X$ is measurable, with $\mu(A)>0$. Then for $\mu$-a.e.\ $x\in X$, for every $\eps>0$, there exists $n_0= n_0(x)$ such that 
$$
\forall\,n\ge n_0,\ \exists\, k \in [n(1-\eps),n]\cap \N:\ \ T^k x \in A.
$$
}
 To verify the latter, it suffices to note that by the Ergodic Theorem,
for a.e.\ $x$ we have $$n^{-1}\cdot \#\{k\in [0,n-1]\cap \N:\ T^k x\in A\}\to \mu(A),\ \ \mbox{ as}\ n\to \infty.$$

Now we state the main result for $S$-adic systems.

%Define Lip-cylindrical functions by the formula: 
%\be \label{fcyl2}
%f(x,t)=\sum_{j\in \Ak} \One_{[j]}(x) \cdot \psi_j(t),\ \ \mbox{with}\ \ \psi_j\in \Lip[0,s_j],
%\ee
%where $X_{\ba} = \bigsqcup_{j\in \Ak} [j] $ is the partition into cylinder sets according to the value of $x_0$.

\begin{theorem} \label{th-main10}
Let $\ba=(\zeta_j)_{j\ge 1}\in \FrA^\N$ be a sequence of substitutions 
%defined on $\Ak = \{1,\ldots,m\}$, 
satisfying the conditions {\bf (A1$'$)}, {\bf (A2)}, and {\bf (A3)}, such that the $S$-adic shift $(X_\ba,T)$ is aperiodic and $\ba$ is recognizable.
Let $\vec{s}\in \R_+^m$ and consider the suspension flow $(\Xx_\ba^{\vec{s}},h_t,\wt \mu)$ over the uniquely ergodic system $(X_\ba,T,\mu)$. 
 Let $f(x,t)=\sum_{j\in \Ak} \One_{[j]}(x)\cdot \psi_j(t)$ be a Lip-cylindrical function and $\sig_f$ the corresponding spectral measure.
 Fix $\om\in \R$ and let 
 $$
\xi = \om\vec{s} \ \mbox{\rm (mod $\Z^m$)},\ \ \vec{z} = (\what{\psi}_j(\om))_{j\in \Ak}.
$$
Suppose that  ${\chi}^+_{\ba,\xi,\vec{z}}> 0$.
Then
 \be \label{ki2}
 \underline{d}(\sig_f,\om) = 2 - \frac{2{\chi}^+_{\ba,\xi,\vec{z}}}{\lam_\ba}\,.
 \ee
 If ${\chi}^+_{\ba,\xi,\vec{z}}\le 0$, then
\be \label{kia}
\underline{d}(\sig_f,\om) \ge 2.
\ee
\end{theorem}

Notice that if ${\chi}^+_{\ba,\xi,\vec{z}}> 0$, then also $\lam_\ba \ge {\chi}^+_{\ba,\xi,\vec{z}}> 0$, so the formula in (\ref{ki2}) is well-defined.

\smallskip

As a special case, we can consider a single primitive aperiodic substitution $\zeta$, and Theorem~\ref{th-main10} applies to suspension flows over the classical substitution dynamical
system $(X_\zeta,T)$, see Section~\ref{subsec-sub} for more details.

The next two corollaries are very similar to the ones from the previous section, but since here the setting is an individual $S$-adic system, we state them  explicitly.
Recall that simple cylindrical functions have the form $f(x,t) = \sum_{a\in \Ak} b_a \One_{[a]}(x)$, that is,
$\psi_a(\om) = b_a\One_{[0,s_a]}$. Then by (\ref{cylia}):
$$
\vec{z} = (\what{\psi}_a(\om))_{a\in \Ak} = \left(\frac{b_a(1-e^{-2\pi i \om s_a})}{2\pi i \om}\right)_{a\in \Ak}=: \Gam_\om \vec{b},\ \ \mbox{for}\ \om\ne 0,
$$
and for $\om=0$ we have $\vec{z} = \Gam_0(\vec{b})=\vec{b}$.
Let $\vec{e}_j$ be the $j$-th unit coordinate vector in $\C^m$.

\begin{corollary} \label{cor-spec00} Let $\ba=(\zeta_j)_{j\ge 1}\in \FrA^\N$ be a sequence of substitutions 
%defined on $\Ak = \{1,\ldots,m\}$, 
satisfying the conditions {\bf (A1$'$)}, {\bf (A2)}, and {\bf (A3)}, such that the $S$-adic shift $(X_\ba,T)$ is aperiodic and $\ba$ is recognizable.
Let $\vec{s}\in \R_+^m$ and consider the suspension flow $(\Xx_\ba^{\vec{s}},h_t,\wt \mu)$ over the uniquely ergodic system $(X_\ba,T,\mu)$. 

{\rm (i)} Let $\om\in \R\setminus \bigcup_{a\in\Ak} (s_a^{-1}\cdot \Z)$.  For any $\vec{b}\in \C^m$ there exists $j\le m$ such that
$$
\card\Bigl\{c\in \C:\ \chi^+_{\ba,\om\vec{s},\, \Gam_\om(\vec{b} + c\vec{e}_j)}< \chi^+_{\ba,\om\vec{s}}\Bigr\} \le 1,
$$
and hence for all perturbations of $\vec{b}$ along the $j$-th direction, except possibly one, the simple cylindrical function $f$ corresponding to it, satisfies
$$
\und{d}(\sig_f,\om) = 2 - \frac{2\chi^+_{\ba,\om\vec{s}}}{\lam_\ba}\,, \  \mbox{assuming}\  \chi^+_{\ba,\om\vec{s}}>0.
$$

{\rm (ii)} For Lebesgue-a.e.\ $\vec{b}\in \C^m$, the simple cylindrical function $f$ corresponding to $\vec{b}$ satisfies
$$
\und{d}(\sig_f,\om) = 2 - \frac{2\chi^+_{\ba,\om\vec{s}}}{\lam_\ba}\ \ \mbox{for Lebesgue-a.e.}\ \om\in \R, \ \mbox{such that}\  \chi^+_{\ba,\om\vec{s}}>0.
$$
\end{corollary}

%Observe that $\chi_{\om\vec{s}}^+\ge 0$ is immediate from the definition of the cocycle.

\begin{corollary} \label{cor-spec10} Let $\ba=(\zeta_j)_{j\ge 1} \in \FrA^\N$ be a sequence of substitutions 
%defined on $\Ak = \{1,\ldots,m\}$, 
satisfying the conditions {\bf (A1$'$)}, {\bf (A2)}, and {\bf (A3)}, such that the $S$-adic shift $(X_\ba,T)$ is aperiodic and $\ba$ is recognizable.
Let $\vec{s}\in \R_+^m$.
Then

{\rm (i)} we have $\chi_{\ba,\om\vec{s}}^+ 
\le  \half\lam_\ba$ for Lebesgue-a.e.\ $\om\in \R$, but $\chi_{\ba,\om\vec{s}}^+ 
\ge  \half\lam_\ba$ for $\sig_f$-a.e.\ $\om\in \R$.

{\rm (ii)}
if $\chi_{\ba,\om\vec{s}}^+ <  \half\lam_\ba$ for Lebesgue-a.e.\ $\om\in \R$, 
then for any cylindrical function $f$ on $\Xxi^{\ba}$, the spectral measure $\sig_f$ is purely singular;

{\rm (iii)}
if $\chi_{\ba,\om\vec{s}}^+ <  \half\lam_\ba$ for Lebesgue-a.e.\ $\om\in \R$, and moreover, $\om\mapsto \det\bigl(\Cc_\ba(\om \vec{s},\ell)\bigr)$ is not constant zero
for all $\ell$,
then the flow $(\Xx_\ba^{\vec{s}}, {\widetilde \mu}, h_t)$ has purely singular spectrum.
\end{corollary}

Corollaries \ref{cor-spec00} and \ref{cor-spec10} will be proven in Section~\ref{sec-add}.
A sufficient condition for $\det\bigl(\Cc_\ba(\om \vec{s},\ell)\bigr)\not\equiv 0$ is
$\det(\Sf^{[\ell]})\ne 0$, since then $\det\bigl(\Cc_\ba(\om \vec{s},\ell)\bigr)$ is a non-trivial trigonometric polynomial.

%%%%%%%%%%%%%%%%%%%%%%%%%%%%%%%%%%%%%

\subsection{Recognizability and higher-level cylindrical functions} As already mentioned, we need higher-order cylindrical functions to describe the spectral type of the flow.
Recognizability of $\ba$  implies that
\be \label{KR}
\Pk_n = \{T^i(\zeta^{[n]}[a]):\ a\in \Ak,\ 0 \le i < |\zeta^{[n]}(a)|\}
\ee
is a sequence of Kakutani-Rokhlin partitions for $n\ge n_0(\ba)$, which generates the Borel $\sig$-algebra on the space $X_\ba$. We emphasize that, in general, $\zeta^{[n]}[a]$ may be a proper subset of $[\zeta^{[n]}(a)]$.

Using the uniqueness of the representation (\ref{recog}) and the Kakutani-Rokhlin partitions (\ref{KR}), we obtain for $n\ge n_0$:
$$
\mu([a]) = \sum_{b\in \Ak}\Sf^{[n]}(a,b)\,\mu(\zeta^{[n]}[b]),\ \ a\in \Ak,
$$
hence
\be \label{eq-measure}
\vec{\mu}_0 = \Sf^{[n]}\vec{\mu}_n,\ \ \mbox{where}\ \ \vec{\mu}_n = \bigl(\mu(\zeta^{[n]}[b])\bigr)_{b\in \Ak}
\ee
is a column-vector. Similarly to (\ref{eq-measure}), we have that 
\be \label{eq-measure2}
\vec{\mu}_n = \Sf_{n+1}\vec{\mu}_{n+1},\ \  n\ge n_0. 
\ee

Let $\ell\in \N$.
It follows from recognizability that the suspension flow $(\Xx_{\ba}^{\vec{s}},\wt{\mu},h_t)$ is measurably isomorphic to the suspension flow over the system $(\zeta^{[\ell]}(X_{\sig^\ell\ba}), \zeta^{[\ell]}\circ T\circ (\zeta^{[\ell]})^{-1})$, with the induced measure, and a piecewise-constant roof function given by the vector
$$
\vec{s}^{\,(\ell)}= (s^{(\ell)}_a)_{a\in \Ak}:= (\Sf^{[\ell]})^t \vec{s}.
$$
(If we think about the suspension flow as a tiling system, this procedure corresponds to considering ``supertiles'' of order $\ell$.)
We have a union, disjoint in measure:
$$
\Xx_\ba^{\vec{s}} = \bigcup_{a\in \Ak} \zeta^{[\ell]}[a]\times [0, s_a^{(\ell)}].
$$
Note that this works correctly in terms of total measure of the space: by (\ref{eq-measure}) we have 
$$
\wt \mu(\Xx_\ba^{\vec{s}}) = \langle \vec{\mu}_0,\vec{s} \rangle = \langle \Sf^{[\ell]}\vec{\mu}_\ell, \vec{s}\rangle = \langle \vec{\mu}_\ell, \vec{s}^{(\ell)}\rangle.
$$
Correspondingly,  define $f$ to be a Lip-cylindrical function of level $\ell$ if 
$$
f(x,t) = \psi_a^{(\ell)}(t),\ \ x\in \zeta^{[\ell]}[a],\ 0 \le t < s^{(\ell)}_a,
$$
for some $\psi_a\in \Lip[0,s^{(\ell)}_a], \ a\le m$.
The union of cylindrical functions of all levels is dense in $L^2(\Xx_\ba^{\vec{s}})$, see Lemma~\ref{lem-approx} below, hence the maximal spectral type of the flow may be expressed in terms of their spectral measures.

For a generalization of Theorem~\ref{th-main10} to the case of cylindrical function of level $\ell$ we need a version of condition (A2):

\medskip

{\bf (A2-$\ell$)} $$\lim_{n\to \infty} \frac{1}{n} \log\|\Sf_{\ell+1}\cdot \ldots\cdot \Sf_{\ell+n}\| = \lam_\ba.$$

\medskip

Observe that condition (A2-$\ell$) follows from (A2) if the substitution matrices are invertible, in view of the inequalities
$$
\|\Sf_1\cdot\ldots \cdot\Sf_{\ell+n}\|\cdot \|\Sf_1\cdot\ldots\cdot \Sf_\ell\|^{-1} \le \|\Sf_{\ell+1}\cdot \ldots\cdot \Sf_{\ell+n}\| \le \|(\Sf_1\cdot\ldots \cdot\Sf_\ell)^{-1}\|\cdot \|\Sf_1\cdot\ldots\cdot \Sf_{\ell+n}\|.
$$

\begin{theorem} \label{th-main1110}
Let $\ba=\zeta_1\,\zeta_2,\ldots$ be a sequence of substitutions defined on
$\Ak = \{1,\ldots,m\}$, 
satisfying the conditions {\bf (A1$'$)}, {\bf (A2-$\ell$)}, and {\bf (A3)}, such that the $S$-adic shift $(X_\ba,T)$ is aperiodic and $\ba$ is recognizable.
Let $\vec{s}\in \R_+^m$ and consider the suspension flow $(\Xx_\ba^{\vec{s}},h_t,\wt \mu)$ over the $S$-adic system $(X_\ba,T,\mu)$.   For $\ell\in \N$ let $f(x,t)=\sum_{j\in \Ak} \One_{\zeta^{[j]}[a]}(x) \cdot \psi^{(\ell)}_j(t)$ be a Lip-cylindrical function of level $\ell$ and $\sig_f$ the corresponding spectral measure.
 Fix $\om\in \R$ and let $\xi\in \T^m$ be such that 
$$
\xi = \om \vec{s}^{(\ell)}\ (\mbox{\rm mod}\ \Z^m).
$$
Further, define
$$
\vec{z} = (\what{\psi}^{(\ell)}_j(\om))_{j\in \Ak}.
$$
Suppose that  ${\chi}^+_{\ba,\xi,\vec{z}}> 0$.
Then
 \be \label{ki20}
 \underline{d}(\sig_f,\om) = 2 - \frac{2{\chi}^+_{\ba,\xi,\vec{z}}}{\lam_\ba}\,.
 \ee
 If ${\chi}^+_{\ba,\xi,\vec{z}}\le 0$, then
\be \label{kia0}
\underline{d}(\sig_f,\om) \ge 2.
\ee
 \end{theorem}
 
 This theorem has corollaries that are exact analogues of Corollaries~\ref{cor-spec00} and \ref{cor-spec10} for the higher-level cylindrical functions.

  %%%%%%%%%%%%%%%%%%%%%%%%%%%%
  
 \subsection{Case of a single substitution} \label{subsec-sub}
As a special case of an $S$-adic system, we can consider $\ba=(\zeta_n)_{n\ge 1}$, where $\zeta_n\equiv\zeta$, a fixed primitive aperiodic substitution on $\Ak$. 
Thus we obtain results on the spectral measures of substitution $\R$-actions as corollaries. We note that suspension flows in this case have been studied as tiling dynamical systems on the line,
with interval prototiles of length $s_a$, $a\in \Ak$, see \cite{BR,CS,SolTil}.

Assuming $\det (\Sf_\zeta) \ne 0$, consider the toral endomorphism 
\begin{equation}\label{torend}
E_\zeta: \xi \mapsto \Sf_\zeta^t \xi \  (\mathrm{mod} \   \Z^m),
\end{equation} 
 which preserves the Haar measure. Then the matrix-function from Definition~\ref{def-cocycle} becomes
\be \label{cocycle3}
\Cc_\zeta(\xi,n):= \Cc_\zeta\bigl((E_\zeta)^{n-1}\xi \bigr)\cdot \ldots \cdot \Cc_\zeta(\xi) = \Cc_{\zeta^n}(\xi)
\ee
and forms a complex matrix cocycle over the endomorphism \eqref{torend}. All the conditions (A1$'$), (A2), (A3) clearly hold, so we obtain Theorems~\ref{th-main10}, \ref{th-main1110}, and the corollaries, specialized to the case of a single substitution. Observe that now
$$
\lam_\ba = \log\theta_1,
$$
where $\theta_1$ is the Perron-Frobenius eigenvalue of the substitution matrix. 
In the case of a single substitution we denote the pointwise upper Lyapunov exponents by $\chi^+_{\zeta,\xi}$.

We note that in \cite[Prop.\,7.2]{BuSo1} a lower bound for the local dimension of $\sig_f$ is obtained for substitution $\Z$-actions, which is similar in spirit to the lower bound in (\ref{ki2}) in 
the single substitution case.

Now suppose that $\Sf_\zeta$ has no eigenvalues that are roots of unity. Then the  endomorphism $E_\zeta$ of the torus $\T^m$ induced by $\Sf_\zeta^t$ is ergodic, and
Furstenberg-Kesten Theorem \cite{FK} yields that the top Lyapunov exponent exists and is constant almost everywhere:
$$
\exists\, \chi(\Cc_\zeta) = \lim_{n\to\infty} \frac{1}{n}\log\|\Cc_\zeta(\xi,n)\|\ \ \ \mbox{for a.e.}\ \xi\in \T^m.
$$

\begin{corollary} \label{cor-ac}
Suppose that the toral endomorphism $E_\zeta$ is ergodic. Then

{\rm (i)} the top Lyapunov exponent of the spectral cocycle satisfies
$
\chi(\Cc_\zeta)  \le \half\log\theta_1;
$

{\rm (ii)} if $\chi(\Cc_\zeta) < \half\log\theta_1$, 
then  for a.e.\ $\vec{s}\in \R^m_+$, the flow $(\Xxi^{\vec{s}}, {\widetilde \mu}, h_t)$ has purely singular spectrum.
\end{corollary}

This corollary is immediate from Corollary~\ref{cor-spec10}. Indeed, the condition on the determinant in part (iii) of that corollary holds automatically, since $\det(\Sf_\zeta)\ne 0$ by assumption. Note that, even in the ergodic case, Corollary~\ref{cor-spec10} is stronger, since it makes a claim about the flow under {\em every} piecewise-constant roof function, and not just almost every, as follows from Corollary~\ref{cor-ac} and Fubini's Theorem. It would be interesting to find explicit examples of purely singular spectrum using Corollaries \ref{cor-spec10} and \ref{cor-ac}.
 
\begin{remark}
 {\em (i) Corollaries~\ref{cor-spec00} and \ref{cor-spec10} provide information about dimension of spectral measures and pointwise Lyapunov exponents, as well as a sufficient condition for singularity, for a (geometric) substitution tiling on the line corresponding to the substitution $\zeta$ and the tile lengths given by an arbitrary vector $\vec s\in \R^m_+$. It is not hard to show that if $\vec s$ and $\vec s'$ belong to the same {\em stable manifold} of $E_\zeta$, that is, if
\be \label{eq-stable}
\lim_{n\to \infty} E_\zeta^n(\vec s - \vec s') = 0,
\ee
then the corresponding pointwise Lyapunov exponents are equal: $\chi^+_{\zeta,\om \vec s} = \chi^+_{\zeta,\om \vec s'}$, for any $\om \in \R$, and hence we have an equality for the dimension of spectral measures. In fact, by a theorem of Clark and Sadun \cite[Theorem 3.1]{CS} the two suspension flows $(\Xxi^{\vec{s}}, {\widetilde \mu}, h_t)$  and
$(\Xxi^{\vec{s}'}, {\widetilde \mu}, h_t)$ are topologically conjugate, provided (\ref{eq-stable}) holds.

(ii) The {\em Pisot case}, when all the eigenvalues of $\Sf_\zeta$, other than $\theta_1$, lie inside the unit circle,
has attracted much attention of the researchers. It is still an open problem, whether in the general Pisot case the spectrum is pure discrete, but it is known that there is a large discrete spectrum, dense on the line. Sadun \cite{Sadun16} called a substitution ``plastic'' if arbitrary changes of tile lengths (vector $\vec s\in \R^m_+$) result in a topologically conjugate system, up to an overall scale. Pisot substitutions are plastic by the Clark-Sadun \cite[Theorem 3.1]{CS}, and more generally, ``homological Pisot substitutions'' of \cite{BBJS}.
}
\end{remark} 
 
 \subsection{Self-similar suspension flow; comparison with the work of Baake et al.} \label{subsec-Baake}

Consider the {\em self-similar} suspension flow over a substitution, when the vector $\vec{s}$ defining the roof function is given by the Perron-Frobenius eigenvector of $\Sf_\zeta^t$. In particular, this is the case of translation flows along stable/unstable foliations for a pseudo-Anosov diffeomorphism. Then one gets
\be \label{coc2}
\Cc_\zeta(\om \vec{s},n):= \Cc_\zeta\bigl(\theta_1^{n-1}\om \vec{s}\bigr)\cdot \ldots \cdot \Cc_\zeta(\theta_1\om\vec{s})\cdot \Cc_\zeta(\om \vec{s}).
\ee

Baake et al.\ \cite{BFGR,BGM,BG2M} studied the diffraction spectrum of substitution systems in the self-similar non-constant length case, with a goal of proving that the spectrum is pure singular.
There is a  well-known connection between the diffraction and dynamical spectrum of a system, see \cite{Dworkin,Hof2,BLE}, so that the questions are closely related. 
Baake et al.\ 
work with  (\ref{coc2}) as a cocycle on $\R$ over the infinite-measure preserving action $\om\mapsto \theta_1 \om$. They represent the diffraction measure as a matrix Riesz product, using a weak-star limit of (\ref{coc2}), appropriately normalized and viewed as a matrix of densities of absolutely continuous measures.
(We note that a similar generalized Riesz product expression for the matrix of spectral measures for a substitution $\Z$-action was given in \cite[Lemma 2.2]{BuSo1}.)
Building on this expression, Baake et al.\ obtain functional equations for the absolutely continuous components of the diffraction measures, which under certain conditions, expressed in terms of the upper Lyapunov exponents of (\ref{coc2}), imply triviality of these components and hence singularity of the spectrum.
Earlier this approach was used in the constant length case, in particular, for the Thue-Morse substitution and its generalizations \cite{KS,BGG}. In comparison, our proof of singularity is based on growth estimates of twisted Birkhoff integrals, which then yield pointwise dimensions of spectral measures incompatible with a non-trivial absolutely continuous component.

To be more specific, in the paper \cite{BFGR} estimates for the upper Lyapunov exponent of (\ref{coc2}) at almost every $\om\in \R$ were used to establish singularity of the diffraction measure 
for a  non-Pisot substitution on a two-letter alphabet, defined by  $0\to 0111,\ 1\to 0$, and in \cite{BGM} this was extended to the family of substitutions $0\to 01^k,\ 1\to 0$, with $k\ge 4$. In the paper \cite{BG2M} a general primitive substitution is considered. \cite[Theorem 3.24]{BG2M} may be restated as a claim that if $\chi^+_{\zeta,\om\vec{s}}< \half\log\theta_1$, for Lebesgue-a.e.\ $\om\in \R$, then the diffraction measure is singular. This can be deduced from Corollary~\ref{cor-spec10}(ii). % It should be noted that, although the renormalization %techniques of  \cite{BFGR,BGM,BG2M} have common features with our methods in \cite{BuSo1} and the current paper, their approach to proving singularity is quite different. 
In the special case of self-similar flow,  Corollary~\ref{cor-spec10}(i) is proved in \cite[Theorem 3.29]{BG2M} directly. 

The almost sure existence  of the Lyapunov exponent for (\ref{coc2}) is not clear, except in the Pisot case and in the constant length case, see \cite{BG2M}.
 Considering the cocycle (\ref{cocycle3}) on the torus we automatically obtain the existence of the Lyapunov exponent in the ergodic case
 almost everywhere on $\T^m$, but, of course, the set $\{\om \vec{s}\  \mbox{(mod  $\Z^m$)}:\ \om\in \R\}$ for a fixed $\vec{s}$ has Haar measure zero.

It should be mentioned that Queffelec \cite{Queff} made extensive use of generalized matrix Riesz products to represent spectral measures for constant length substitutions.  Various  classes of Riesz products appeared in the description of the spectral type of measure-preserving systems in other contexts, see e.g.\ \cite{Led,Bour,DE,Abd}.

\bigskip
 
 The rest of the paper is organized as follows. In the next section we obtain estimates relating the growth of twisted Birkhoff integrals to the local behaviour of spectral measures. It is completely general and applies to any measure-preserving flow. In Section~\ref{section-proofs} we prove Theorems~\ref{th-main10} and \ref{th-main1110}, and then derive Theorems~\ref{th-main1} and \ref{th-main11}.
Section~\ref{sec-add} contains the remaining proofs. Finally, in the Appendix, Section~\ref{section-stratum}, we verify H\"older continuity for a.e.\ translation flow in $\Hk(1,1)$.

%%%%%%%%%%%%%%%%%%%%%%%%%%%%%%%%%%%%%%%%%%%%%%%%%%%%%%%%%%%%%%%%

\section{Spectral estimates} \label{section-spectral}

Let $(Y,\mu,h_t)$ be a measure-preserving flow.
For $f \in L^2(Y,\mu)$, $R>0$,  $\om \in \R$, and $y\in Y$ let
$$
G_R(f,\om) = R^{-1} \left\|\int_0^R e^{-2\pi i \om t} f( h_t y) \,dt\right\|_{L^2(Y)}^2\  \ \mbox{and}\ \  S_R^y(f,\om) = \int_0^R e^{-2\pi i \om t} f(h_t y)\,dt,
$$
so that
$$
G_R(f,\om) = R^{-1} \int_Y |S_R^y(f,\om)|^2\,d\mu(y).
$$
Let $\sig_f$ be the spectral measure for the flow, defined by
$$
\int_\R e^{2\pi i \om t}\,d\sig_f(\om) = \langle f\circ h_t,f\rangle_{_{L^2}},\ \ t\in \R.
$$
It is well-known that 
\be \label{fejer}
G_R(f,\om) = \int_\R K_R(\om-\tau)\,d\sig_f(\tau),\ \ 
 \mbox{ where}\ \ 
 K_R(y) = R^{-1} \left( \frac{\sin (\pi R y)}{\pi y}\right)^2,\ \mbox{for}\ y\ne 0,
 \ee 
 and $K_R(0):=R$,
 is the Fej\'er kernel for $\R$, see e.g.,  \cite[2.2]{Hof}. The following lemma is a minor variation of a result by A. Hof \cite[2.2]{Hof}, and is a special case of \cite[Lemma 3.1]{BuSo1}.

\begin{lemma} \label{lem-spec1}
Suppose that for some fixed $\om \in \R$, $\alpha \ge 0$, { fixed} $R\ge 1$, and $f\in L^2(Y,\mu)$ we have
$$
G_R(f,\om) \le C_1 R^{1-\alpha}.
$$
Then
$$
\sig_f(B_r(\om)) \le C_1 \pi^2 2^\alpha  r^\alpha,\ \ \mbox{for}\ r = (2R)^{-1}.
$$
\end{lemma}

\begin{proof}
Note that $|\sin x| \ge (2/\pi)|x|$ for $|x|\le \pi/2$, hence $K_R(y)\ge 4R/\pi^2$ for $|y|\le (2R)^{-1}$. Thus, (\ref{fejer}) implies $\sig_f(B_{\frac{1}{2R}}(\om)) \le \frac{\pi^2}{4R} G_R(f,\om)$, and the desired estimate follows.
\end{proof}

The next lemma is also inspired by \cite{Hof}, however, there is a difference, since we are interested in estimates for a fixed $\om$, whereas A. Hof considers uniform H\"older continuity on an interval.

\begin{lemma} \label{lem-specest}
Suppose that for some fixed $\om \in \R$, $\alpha\in (0,2)$, $r_0\in (0,1)$, and $f\in L^2(Y,\mu)$ we have
\be \label{est1}
\sig_f(B_r(\om)) \le C_2 r^\alpha,\ \ \mbox{for all}\ \ r\in (0,r_0).
\ee
Then
\be \label{est2}
G_R(f,\om) \le C_3 R^{1-\alpha},\ \ \mbox{for all}\ \ R\ge r_0^{-\frac{2}{2-\alpha}},
\ee
where one can take
$$
C_3 = C_2\Bigl(1 + \frac{2}{\pi^2(2-\alpha)}\Bigr) + \pi^{-2}{\|f\|}_2^2.
$$
\end{lemma}

The ``borderline case'' of $\alpha=2$ is considered in the next lemma. The result is similar, but with a logarithmic correction.

\begin{lemma} \label{lem-spec2}
Suppose that for some fixed $\om \in \R$,  $r_0\in (0,e^{-1})$, and $f\in L^2(Y,\mu)$ we have
\be \label{est11}
\sig_f(B_r(\om)) \le \wt C_2 r^2,\ \ \mbox{for all}\ \ r\in (0,r_0).
\ee
Then
\be \label{est21}
G_R(f,\om) \le \wt C_3 R^{-1}\ln R,\ \ \mbox{for all}\ \ R\ge \max\{r_0^{-1}, e^{r_0^{-2}}\},
\ee
where one can take
$$
\wt C_3 = \wt C_2(1 + 2\pi^{-2}) + \pi^{-2}{\|f\|}_2^2.
$$
\end{lemma}

Observe that
$$
K_R(y) \le \left\{\begin{array}{ll} R, & |y| \le \frac{1}{R};\\[1.2ex]
                                                  \frac{1}{\pi^2 R y^2}, & |y| > \frac{1}{R}. \end{array} \right.
$$
It follows that
\be \label{est3}
G_R(f,\om) \le R \cdot \sig_f(B_{1/R}(\om)) + \frac{1}{\pi^2 R} \int_{|\tau-\om|> 1/R} \frac{d\sig_f(\tau)}{(\om-\tau)^2}\,.
\ee

\begin{proof}[Proof of Lemma~\ref{lem-specest}]
Let $R\ge r_0^{-\frac{2}{2-\alpha}}$. By (\ref{est1}), the first term in (\ref{est3}) is not greater than $C_2 R^{1-\alpha}$, since $R^{-1} \le  r_0^{\frac{2}{2-\alpha}} \le r_0$. The integral in the second term can be rewritten, and then estimated, as

\smallskip

\begin{eqnarray*}
\int_{|\tau-\om|> 1/R} \frac{d\sig_f(\tau)}{(\om-\tau)^2} &  = & \int_0^{R^2} \sig_f\bigl(\{\tau:\ R^{2}> |\tau-\om|^{-2} \ge t\}\bigr)\,dt \\[1.2ex]
                                                                                     &  =  & 2\int_{1/R}^{\infty} \sig_f\bigl(\{\tau:\ R^{2} > |\tau-\om|^{-2}  \ge x^{-2}\}\bigr)\,x^{-3}\,dx \\[1.2ex]
                                                                                     &  =  & 2\int_{1/R}^{\infty} \sig_f\bigl(\{\tau:\ R^{-1} < |\tau-\om|  \le  x\}\bigr)\,x^{-3}\,dx \\[1.2ex]
                                                                                     &  \le & 2\int_{1/R}^\infty \sig_f(B_x(\om))\,x^{-3}\,dx \\[1.2ex]
                                                                                     &  = & 2\left(\int_{1/R}^{R^{-\frac{2-\alpha}{2}}}  + \int_{R^{-\frac{2-\alpha}{2}}}^\infty \right)\sig_f(B_x(\om))\,x^{-3}\,dx.                                                                                                                                                                      \end{eqnarray*}
The first integral in the last line may be estimated from above by 
$$
2\int_{1/R}^{R^{-\frac{2-\alpha}{2}}} \sig_f(B_x(\om))\,x^{-3}\,dx\le  2C_2 \int_{1/R}^\infty x^{\alpha-3}\,dx = \frac{2C_2 R^{2-\alpha}}{2-\alpha}.$$
 In the second integral, we simply use that
$\sig_f$ is a positive measure of total mass ${\|f\|}^2_2$ to obtain
$$
2\int_{R^{-\frac{2-\alpha}{2}}}^\infty \sig_f(B_x(\om))\,x^{-3}\,dx \le 2{\|f\|}^2_2 \int_{R^{-\frac{2-\alpha}{2}}}^\infty x^{-3}\,dx = {\|f\|}^2_2\cdot R^{2-\alpha}.
$$
Combining these together with (\ref{est3}) yields (\ref{est2}).
\end{proof}

\begin{proof}[Proof of Lemma~\ref{lem-spec2}]
Let $R\ge \max\{r_0^{-1}, e^{r_0^{-2}}\}$.
By (\ref{est11}), the first term in (\ref{est3}) is not greater than $\wt C_2 R^{-1}\le \wt C_2 R^{-1}\log R$, where we used $R \ge r_0^{-1} \ge e$. The integral in the second term of (\ref{est3}) can be estimated as follows (the change of variable is the same as in the proof of Lemma~\ref{lem-specest}):
\begin{eqnarray*}
\int_{|\tau-\om|> 1/R} \frac{d\sig_f(\tau)}{(\om-\tau)^2}
                                                                                     &  \le & 2\int_{1/R}^\infty \sig_f(B_x(\om))\,x^{-3}\,dx \\[1.2ex]
                                                                                     &  = & 2\left(\int_{1/R}^{1/\sqrt{\log R}}  + \int_{1/\sqrt{\log R}}^\infty \right)\sig_f(B_x(\om))\,x^{-3}\,dx.                                                                                                                                                                      \end{eqnarray*}
Note that 
$\sqrt{\log R}\le r_0$, so the first integral above  can be estimated from above using (\ref{est11}): 
$$
2 \int_{1/R}^{1/\sqrt{\log R}}\sig_f(B_x(\om))\,x^{-3}\,dx \le 2\wt C_2 \int_{1/R}^{1} x^{-1}\,dx = 2\wt C_2 \log R.
$$
For the second integral, we have
$$
2  \int_{1/\sqrt{\log R}}^\infty \sig_f(B_x(\om))\,x^{-3}\,dx \le 
2{\|f\|}^2_2 \int_{1/\sqrt{\log R}}^\infty x^{-3}\,dx=2{\|f\|}^2_2\cdot \log R.
$$
Combining these together with (\ref{est3}) yields (\ref{est21}).                                                                                     
\end{proof}

The following remark is (essentially) taken from \cite[2.5]{Hof}.

\begin{remark} {\em 
(i) For all $\om\in \R$ we have $\sig_f(\{\om\}) = \lim_{R\to\infty} R^{-1} G_R(f,\om)$, since $R^{-1}K_R$ tends to zero uniformly outside any neighborhood of $0$, and $R^{-1}K_R(0)=1$.

(ii) Let $d\sig_f = h\,dx + d(\sig_f)_{\rm sing}$ be the decomposition of the spectral measure into the absolutely continuous and singular parts. Then
$$
\lim_{R\to \infty} G_R(f,\om) = h(\om)\ \ \mbox{for Lebesgue-a.e.}\ \om\in \R.
$$
}
\end{remark}

%%%%%%%%%%%%%%%%%%%%%%%%%%%%%%%%%%%%%%%%%%%%%%%%%%%%%%%%%%%%%%%%%%%%%%%

\section{Proof of the theorems} \label{section-proofs}

\subsection{Twisted ergodic integrals and the spectral cocycle}
Using the results from the previous section, we can obtain local bounds on the spectral measure from estimates of the growth of twisted ergodic integrals
\be \label{twist1}
S_R^{(y)}(f,\om) := \int_0^R e^{-2\pi i \om \tau} f\circ h_\tau(y)\,d\tau.
\ee

Recall that we consider suspension flows over an $S$-adic system, with a roof function determined by $\vec{s} = (s_j)_{j\in \Ak}\in \R^m_+$. Let $f$ be a Lip-cylindrical function. This means, by definition, that there are functions $\psi_j \in \Lip([0,s_j])$ for $j\in \Ak$, and $f = \sum_{j\in \Ak} f_j$, where 
$$
f_j(x,t) = \left\{\begin{array}{ll} \psi_j(t),\ t \in [0,s_j], & \mbox{if}\ x_0 = j;\\ 0, & \mbox{else}.\end{array}\right.
$$
  By definition, $(x,t) = h_t(x,0)$, with $t\in [0, s_{x_0}]$, hence
\begin{eqnarray*}
|S_R^{(x,t)}(f,\om)| & = & \Bigl|\int_0^R e^{-2\pi i \om\tau} f\circ h_{\tau + t}(x,0)\,d\tau \Bigr| \\
& = & \Bigl|\int_t^{R+t} e^{-2\pi i \om\tau} f\circ h_{\tau}(x,0)\,d\tau \Bigr|.
\end{eqnarray*} 
It follows that 
\be \label{er1}
\Bigl| S_R^{(x,t)}(f,\om) - S_R^{(x,0)}(f,\om)\Bigr| \le 2{\|f\|}_\infty \cdot s_{\max},
\ee
where $s_{\max} = \max_{j\in \Ak} s_j$.

For a word $v\in \Ak^+$ denote by $\vec{\ell}(v)\in \Z^m$ its population vector whose $j$-th entry is the number of $j$'s in $v$, for $j\le m$. We will  need the
tiling length of $v=v_0\ldots v_{N-1}$ defined by 
\be \label{tilength}
|v|_{\vec{s}}:= \langle\vec{\ell}(v), \vec{s}\rangle = s_{v_0} + \cdots + s_{v_{N-1}}.
\ee

Fix $x\in X_\ba$ and denote $x[0,N-1] = x_0\ldots x_{N-1}$. For $R\ge s_{\min} = \min_{j\in \Ak} s_j$, let $N\ge 1$ be
maximal such that $\wt R := \left|x[0,N-1]\right|_{\vec{s}} \le R$. Then $|R-\wt R| \le s_{\max}$, hence
\be \label{er2}
\Bigl| S_R^{(x,0)}(f,\om) - S_{\wt R}^{(x,0)}(f,\om)\Bigr| \le {\|f\|}_\infty \cdot s_{\max}.
\ee
The combined error term in (\ref{er1}) and (\ref{er2}) is at most $3{\|f\|}_\infty \cdot s_{\max}$, so it suffices to get bounds on the growth of $|S_{\wt R}^{(x,0)}(f,\om)|$, with 
$\wt R = \left|x[0,N-1]\right|_{\vec{s}}$.

Let $\fu:\Ak\to \C$ be an arbitrary function (this is just a vector in $\C^m$, but it is  sometimes convenient to view it as a function in the alphabet).
For $v=v_0\ldots v_{N-1}\in \Ak^+$ let
\be \label{def-Phi3}
\Phi_\fu^{\vec{s}}(v,\om) = \sum_{j=0}^{N-1} \fu(v_j) \exp(-2\pi i \om |v_0\ldots v_j|_{\vec{s}}).
\ee
In view of
$$
f\circ h_\tau(x,0) = \psi_{x_n} (\tau - (s_{x_0}+\cdots + s_{x_{n-1}})),\ \ \mbox{for}\ \ s_{x_0}+\cdots + s_{x_{n-1}} \le \tau < s_{x_0}+\cdots + s_{x_{n}},
$$
a  calculation yields
\be \label{SR1}
S_R^{(x,0)}(f,\om) = {\Phi_\fu^{\vec{s}}(x[0,N-1],\om)}\ \ \ \mbox{for}\ \ R = \left|x[0,N-1]\right|_{\vec{s}},\ \ \ \mbox{where}\ \ \fu(j) =  \widehat{\psi}_j(\om).
\ee
 In view of (\ref{def-Phi3}),
for any concatenation of two words $uv$ we have
 \be \label{eq-Phi}
 \Phi^{\vec{s}}_\fu(uv,\om) = \Phi^{\vec{s}}_\fu(u,\om) + e^{-2\pi i \om |u|_{\vec{s}}} \cdot\Phi^{\vec{s}}_\fu(v,\om).
 \ee
Denote
$$
\Phi_j^{\vec{s}}(v,\om) = \Phi_{\delta_j}^{\vec{s}}(v,\om),
$$
where $\delta_j(i) = 1$ whenever $i=j$, and else 0.
Observe that
\be \label{matrux}
\left[\Phi_a^{\vec{s}} (\zeta^{[n]}(b), \om)\right]_{(a,b)\in \Ak^2}\cdot{\vec{\fu}} =  \left[\Phi_\fu^{\vec{s}} (\zeta^{[n]}(b), \om)\right]_{b\in \A},
\ee
where $\left[\Phi_a^{\vec{s}} (\zeta^{[n]}(b), \om)\right]_{(a,b)\in \Ak^2}$ is a matrix-function, and
$\vec{\fu}:={(\fu(a))_{a\in \Ak}}$ and the  right-hand side of (\ref{matrux}) are  column vectors.
Suppose that $$\zeta_j(b) = u_1^{b,j}\ldots u_{|\zeta_j(b)|}^{b,j},\ \ b\in \Ak, \ j\ge 1,$$ and let
\be \label{matru}
(\M^{\vec{s}}_{[n]}(\om))(b,c) = \sum_{j \le |\zeta_n(b)|:\ u_j^{b,n} = c} \exp\left[-2\pi i \om \left(|\zeta^{[n-1]}(u_1^{b,n}\ldots u_{j-1}^{b,n})|_{\vec{s}}\right)\right],\ n\ge 1.
\ee
We then obtain, as in   \cite[Section 3.3]{BuSo2}:
\be \label{RRR}
%\Mc_{[n]}^{\vec{s}}(\om):= 
\left[\Phi_a^{\vec{s}} (\zeta^{[n]}(b), \om)\right]_{(a,b)\in \Ak^2} = \M^{\vec{s}}_{[n]}(\om) \cdot \ldots \cdot \M^{\vec{s}}_{[1]}(\om).
\ee
The formula (\ref{matru}) may be rewritten using
\begin{multline}\label{cui1}
|\zeta^{[n-1]}(u_1^{b,n}\ldots u_{j-1}^{b,n})|_{\vec{s}}  =  \langle \ell(\zeta^{[n-1]} (u_1^{b,n}\ldots u_{j-1}^{b,n})), \vec{s}\rangle = \\
 =   \langle \Sf^{[n-1]} \ell(u_1^{b,n}\ldots u_{j-1}^{b,n}), \vec{s}\rangle= \langle \ell(u_1^{b,n}\ldots u_{j-1}^{b,n}), (\Sf^{[n-1]})^t \vec{s}\rangle  =  \sum_{k=1}^{j-1} \bigl(  (\Sf^{[n-1]})^t  \vec{s}\bigr)_{u_k^{b,n}}. 
\end{multline}
Equations (\ref{matru}), (\ref{RRR}), and (\ref{cui1}), together with the Definition  \ref{def-cocycle} of the cocycle, imply
\be \label{new1}
%\Mc_{[n]}^{\vec{s}}(\om) = 
\left[\Phi_a^{\vec{s}} (\zeta^{[n]}(b), \om)\right]_{(a,b)\in \Ak^2} =\Cc_\ba(\xi,n),\ \ n\in \N,\ \ \mbox{for}\ \ \xi = \om \vec{s}\ \mbox{(mod $\Z^m$)}. 
\ee

\medskip

\subsection{Prefix-suffix decomposition and its consequences} 
By the definition of $S$-adic sequence space, for $x\in X_\ba$ and $N\geq 1$, we have
 \begin{equation} \label{eq-accord}
 x[0,N-1] = u_0 \zeta_1(u_1) \zeta^{[2]}(u_2) \ldots \zeta^{[n]} (u_n) \zeta^{[n]} (v_n) \ldots \zeta^{[2]}(v_2)\zeta_1(v_1) v_0,
 \ee
where  $u_j,v_j,\ j=0,\ldots,n$, are respectively proper suffixes and prefixes  of the words $\zeta_{j+1}(b)$, $b\in \mathcal A$. The words $u_j, v_j$ may be empty, but  at least one of $u_n, v_n$ is
nonempty; furthermore, we have
\be \label{Ncond}
\min_{b\in \Ak}|\zeta^{[n]}(b)| \le N \le 2\max_{b\in \Ak} |\zeta^{[{n+1}]}(b)|.
\ee

We will need upper and lower bounds for $N$ in terms  of $n$. 
Denote by $C=C_{\ba,\xi,\eps}>1$ a generic constant, which may depend only on $\ba$, $\xi$, and $\eps$, but may differ from line to line. 
It follows from condition (A2) that for every $\eps>0$ there exists $C=C_{\ba,\eps}$ such that
\be \label{ass1}
C e^{(\lam_\ba-\eps)j}\le {\|\Sf^{[j]}\|}_1 \le C e^{(\lam_\ba+\eps)j}\ \ \ \mbox{for all}\ \ j\in \N.
\ee
Recall that $|\zeta^{[n]}(b)|$ is a column sum of the matrix $\Sf^{[n]}$. 
Thus, (\ref{Ncond}) and (\ref{ass1}) yield, for every $\eps>0$ and $n\in \N$:
\be \label{assN}
N \le 2\max_{b\in \Ak} |\zeta^{[{n+1}]}(b)| = 2\|S^{[n+1]}\|_1 \le C e^{(\lam_\ba+\eps)n}.
\ee

The lower bound for $N$ is obtained in the next lemma, with the help of condition (A1$'$).

\begin{lemma} \label{lem-Nlower}
For  every $x\in X_\ba$ and $\eps>0$ there exists $C=C_{\ba,\eps}>1$, such that for all $N\ge 1$:
\be \label{Nlower}
N \ge C^{-1} e^{(\lam_\ba-\eps)n},
\ee
where $n\in \N$ is from (\ref{eq-accord}).
\end{lemma}

\begin{proof}
We have $N\ge \min_{b\in \Ak} |\zeta^{[n]}(b)|$ by (\ref{Ncond}). Recall that $|\zeta^{[n]}(b)|$ is a column sum of the matrix $\Sf^{[n]}$. We have $\Sf^{[j+1]} = \Sf^{[j]}\Sf_{j+1}$ for all $j\in \N$. Every substitution matrix $\Sf_j= \Sf_{\zeta_j}$ has non-negative integer entries, with at least one positive entry in each column. Thus $\min_{b\in \Ak} |\zeta^{[j+1]}(b)| \ge \min_{b\in \Ak} |\zeta^{[j]}(b)|$ for all $j$. By assumption, the matrix $\Sf_\bq$ has only strictly positive entries, with $|\bq| = |\ell|$, hence
\be \label{nuka3}
{\|\Sf^{[j]}\|}_1 = \max_{b\in \Ak} |\zeta^{[j]}(b)| \le \min_{b\in \Ak} |\zeta^{[j+\ell]}(b)|,\ \ \ \mbox{whenever}\ \ \ba = \zeta_1\ldots\zeta_j \,\bq \,\zeta_{j+\ell+1}\ldots
\ee
Let $\eps'>0$ be such that $(\lam_\ba-\eps')(1-\eps') \ge (\lam_\ba-\eps)$.
By the assumption {\bf (A1$'$)}, we have that there exists
$n_0 = n_0(\eps',\ba)$ such that for all $n\ge n_0$ the word $\bq$ appears as a subword of $\zeta_{\lfloor n-n\eps' \rfloor+1}, \zeta_{\lfloor n-n\eps' \rfloor+2},\ldots, \zeta_n$.
Then by (\ref{nuka3}) and (\ref{ass1}) we have
$$
N \ge \min_{b\in \Ak} |\zeta^{[n]}(b)| \ge C^{-1}  e^{(\lam_\ba-\eps')n(1-\eps')} \ge C^{-1}  e^{(\lam_\ba-\eps)n},\ \ \mbox{for}\ \ n\ge n_0(\eps,\ba),
$$
and by increasing the constant $C$ we can make sure that (\ref{Nlower}) holds for all $N\in \N$.
\end{proof}

Recall that
$\chi_{\ba,\xi,\vec{z}}^+$ is the pointwise  upper Lyapunov exponent of the cocycle at $\xi$ , corresponding to a given vector $\vec{z}$; see (\ref{Lyap1}) for the definition.

\begin{lemma} \label{lem-newup}
For every $\fu:\Ak\to \C$ and every $\eps>0$, we have for all $x\in X_\ba$ and $N\ge 1$,
\be\label{nukaa}
\bigl|\Phi_\fu^{\vec{s}}(x[0,N-1],\om)\bigr| \le C_1 N^{(\chi_{\ba,\xi,\vec{z}}^++\eps)/\lam_\ba}, 
\ee
where $\xi = \om\vec{s}$ (mod $\Z^m$), $\vec{z} = (\phi(a))_{a\in \Ak}$, and $C_1 = C_1(\ba,\xi,\fu,\eps)$.
\end{lemma}

\begin{proof} All constants in the proof may depend on $\ba,\xi,\fu$, and $\eps$, but not on $n$ and $N$. In view of Lemma~\ref{lem-Nlower}, it is enough to show that for
every $\fu:\Ak\to \C$ and every $\eps>0$, we have for all $x\in X_\ba$ and $N\ge 1$,
$$
\bigl|\Phi_\fu^{\vec{s}}(x[0,N-1],\om)\bigr| \le C e^{(\chi_{\ba,\xi,\vec{z}}^++\eps)n},
$$
where $n$ and $N$ are related by (\ref{Ncond}) and $\xi = \om\vec{s}\ ({\rm mod}\ \Z^m)$.
By (\ref{eq-accord}) and (\ref{def-Phi3}), we have
 \begin{eqnarray}
 \bigl|\Phi_\fu^{\vec{s}}(x[0,N-1],\om)\bigr| & \le & \sum_{j=0}^n \Bigl( \bigl| \Phi_\fu^{\vec{s}}(\zeta^{[j]}(u_j),\om)\bigr| + \bigl| \Phi_\fu^{\vec{s}}(\zeta^{[j]}(v_j),\om)\bigr|\Bigr)\nonumber \\
 & \le & 2\sum_{j=0}^n {\|\Sf_{j+1}\|}_1 \cdot \max_{b\in \Ak} \bigl| \Phi_\fu^{\vec{s}}(\zeta^{[j]}(b),\om)\bigr|. \label{eqq}
\end{eqnarray}
Here we used that $|u_j|, |v_j|\le \max_{b\in \Ak} |\zeta_{j+1}(b)| = {\|\Sf_{j+1}\|}_1$.
Now, by (\ref{eqq}), (\ref{new1}), and (\ref{matrux}),
\be \label{nuka}
\bigl|\Phi_\fu^{\vec{s}}(x[0,N-1],\om)\bigr| \le 2 \sum_{j=0}^n {\|\Sf_{j+1}\|}_1 \cdot  {\|\Cc_{\ba}(\xi,j)\vec{z}\|}_1.
\ee
 By the definition of the upper Lyapunov exponent,
 $$
 {\|\Cc_{\ba}(\xi,j)\vec{z}\|}_1 \le C e^{(\chi_{\ba,\xi,\vec{z}}^+ +\eps/2)j},\ \ \ \mbox{for all}\ \ j\ge 0.
 $$
Now we obtain from (\ref{nuka}) and (\ref{ass2}), for appropriate constants:
$$
\bigl|\Phi_\fu^{\vec{s}}(x[0,N-1],\om)\bigr| \le C'\sum_{j=0}^n e^{j \eps/2} e^{(\chi_{\ba,\xi,\vec{z}}^+ +\eps/2)j} \le C'' e^{(\chi_{\ba,\xi,\vec{z}}^++\eps)n},
$$
completing the proof.
\end{proof}

\subsection{Proof of the lower bound in Theorem~\ref{th-main10}}
\begin{sloppypar}
We  will  obtain {\em uniform estimates} of $|S_R^{(x,t)}(f,\om)|$ (which, of course, imply the $L^2$-estimate) from above, and then apply Lemma~\ref{lem-spec1}. By (\ref{SR1}) and (\ref{nukaa}), 
we have for $\eps>0$, $N\ge 1$, and $\wt R = \left|x[0,N-1]\right|_{\vec{s}} \ge N s_{\min}$:
\begin{eqnarray*}
|S_{\wt R}^{(x,0)}(f,\om)| & = & |{\Phi_\fu^{\vec{s}}(x[0,N-1],\om)}| \\
                                & \le & C\cdot N^{(\chi^+_{\ba,\xi,\vec{z}}+\eps)/\lam_\ba}\\
                                & \le & C\cdot (\wt{R}/s_{\min})^{(\chi^+_{\ba,\xi,\vec{z}}+\eps)/\lam_\ba}.
\end{eqnarray*}
Recall that $\vec{z} = (\what{\psi}_a(\om))_{a\in \Ak}$, where $\psi_a$ are the ``components'' of the test function $f$, given by (\ref{fcyl2}), and $\xi = \om\vec{s}$ (mod $\Z^m$). In view of
(\ref{er1}) and (\ref{er2}), the last inequality yields, for $\eps>0$:
$$
|S_R^{(x,t)}(f,\om)| \le C \cdot  {\|f\|}_\infty\cdot  (R/s_{\min})^{(\chi^+_{\ba,\xi,\vec{z}} + \eps)/\lam_\ba} + 3\|f\|_\infty\cdot s_{\max},\ \mbox{for all}\  R\ge s_{\min}.
$$
 Now Lemma~\ref{lem-spec1} implies that for all $\eps>0$,
$$
\sig_f(B_r(\om)) \le \wt C \cdot \max\bigl\{r^{2 - 2(\chi^+_{\ba,\xi,\vec{z}} + \eps)/\lam_\ba},r^2\bigr\},\ \mbox{for all}\ 0 < r \le (2s_{\min})^{-1},
$$
whence 
\be \label{lobound}
\und{d}(\sig_f,\om) \ge \min\Bigl\{2 - \frac{2\chi^+_{\ba,\xi,\vec{z}}}{\lam_\ba},2\Bigr\},
\ee
as desired.
\qed
\end{sloppypar}
%%%%%%%%%%%%%%%%%%%%%%%%%%%%%%%%%%%%%%%

\subsection{Proof of the upper bound in Theorem~\ref{th-main10}}
We need to prove the reverse inequality to (\ref{lobound}), assuming $\chi^+_{\ba,\xi,\vec{z}}>0$, which is equivalent to showing that, for every $\wt\eps>0$,
\be \label{nuinu}
\limsup_{r\to 0} \frac{\sig_f(B_r(\om))}{r^{2-\gam+\wt\eps}} \ge 1,\ \mbox{where}\ \gam = \frac{2\chi^+_{\ba,\xi,\vec{z}}}{\lam_\ba}\,.
\ee
Suppose that (\ref{nuinu}) is false for some $\wt\eps>0$. Then $\sig_f(B_r(\om))\le r^{2-\gam+\wt\eps}$ for all $r>0$ sufficiently small.
Thus we are in a position to apply Lemma~\ref{lem-specest}, which yields
$$
G_R(f,\om) = R^{-1} \int_{\Xx_{\ba}^{\vec{s}}} |S_R^{(x,t)}(f,\om)|^2\,d\wt{\mu}(x,t) \le C'\cdot R^{1-(2-\gam+\wt\eps)},
$$
or equivalently,
\be \label{byaka1}
 \int_{\Xx_{\ba}^{\vec{s}}} |S_R^{(x,t)}(f,\om)|^2\,d\wt{\mu}(x,t) \le C'\cdot R^{\gam-\wt\eps},
\ee
for all $R>0$ sufficiently large.
In order to get a contradiction, it is enough to get an appropriate lower bound for $ |S_R^{(x,t)}(f,\om)|$ which holds on a set of positive measure, and this measure is estimated from below. This will be our goal.

\medskip

By the definition of the pointwise upper  Lyapunov exponent of the cocycle at $\xi$, corresponding to a  vector $\vec{z}$,
for any $\eps>0$  there exists an infinite set $\Nk\subset \N$ such that
$$
{\|\Cc_{\ba}(\xi,n)\vec{z}\|}_1 \ge e^{(\chi_{\ba,\xi,\vec{z}}^+ - \eps)n},\ \ \forall\,n\in \Nk.
$$
In view of (\ref{new1}), replacing $\Nk$ by a subsequence if necessary, we obtain that
\be \label{byaka2}
\exists\,b\in \Ak:\ \ \left|\Phi_\fu^{\vec{s}} (\zeta^{n}(b), \om)\right| \ge e^{(\chi_{\ba,\xi,\vec{z}}^+- \eps)n},\ \ \forall\,n\in \Nk.
\ee

\begin{lemma} \label{lem-measest}
Suppose that the sequence of substitutions $\ba$ satisfies (A1) and (A2). Then for any $\eps>0$ there exists $n_0 = n_0(\ba,\eps)$ such that for all $b\in \Ak$,
\be \label{byaka3}
\mu\bigl(\zeta^{[n]}[b]\bigr) \ge e^{-(\lam_\ba+\eps)n},\ \ \forall\ n\ge n_0.
\ee
\end{lemma}

\begin{proof}
%We have
%$$
%\zeta^{[n+1]}[b] = \zeta^{[n]}\zeta_{n+1}[b].
%$$
Recall that $\vec{\mu}_0 = \Sf^{[n]}\vec{\mu}_n$, see (\ref{eq-measure}). Thus,
$$
1 = {\|\vec{\mu}_0\|}_1 = {\|\Sf^{[n]}\vec{\mu}_n\|}_1 \le \|\Sf^{[n]}\|\cdot {\|\vec{\mu}_n\|}_1,
$$
and the assumption (A2), or (\ref{ass2}), yields, for $n$ sufficiently large:
\be \label{inuha}
{\|\vec{\mu}_n\|}_1 \ge \|\Sf^{[n]}\|^{-1} \ge e^{-(\lam_\ba+\eps)n}.
\ee
Recall also that $\vec{\mu}_n = \Sf_{n+1}\vec{\mu}_{n+1}$ for $n\ge n_0$, see (\ref{eq-measure2}). Now it follows from the assumption (A1) that 
$$
\vec{\mu}_n \in \Sf_{\bq}\R^m_+,
$$
for $n$ sufficiently large, where $\Sf_\bq$ is a fixed matrix having strictly positive entries and $\R^m_+$ is the cone of non-negative vectors. Therefore, the ratios of components of $\vec{\mu}_n$ are uniformly bounded, and hence (\ref{inuha}) yields the desired (\ref{byaka3}).
\end{proof}

Fix $b\in \Ak$ from (\ref{byaka2}). We will take 
$$
R = R_n={|\zeta^{[n]}(b)|}_{\vec{s}},\ \ n\in \Nk.
$$
This will be our sequence $R_n\to\infty$ for which we will obtain a lower bound for $G_{R_n}(f,\om)$. 
For $n\in \Nk$  and $\tau\in [0,R_n)$ consider the set
\be \label{byaka35}
E_{n,\tau}:= \bigcup_{t\in [0,\tau]}h_t \Bigl(\bigl\{(x,0):\ x\in \zeta^{[n]}[b] \bigr\}\Bigr)\subset \Xx^{\vec{s}}_\ba
\ee
(recall that $\{h_t\}_{t\in \R}$ is our suspension flow).
From  recognizability of $\ba$ it follows that the union in (\ref{byaka35}) is disjoint, and we obtain by the definition of the measure $\wt\mu$ and Lemma~\ref{lem-measest}:
\be \label{byaka4}
\wt\mu(E_{n,\tau}) = \tau \cdot \mu\bigl(\zeta^{[n]}[b]\bigr) \ge \tau\cdot e^{-(\lam_\ba+\eps)n}.
\ee

Suppose that $(x,t)\in E_{n,\tau}$.
Recall that (\ref{er1}) and (\ref{er2}) imply 
$$
\Bigl| S_{R_n}^{(x,t)}(f,\om) - S_{\wt R}^{(x,0)}(f,\om)\Bigr| \le 3{\|f\|}_\infty \cdot s_{\max},
$$
where  $\wt R = \left|x[0,N-1]\right|_{\vec{s}}$ and $N\in \N$ is maximal for which $\wt R\le R_n$. By (\ref{SR1}),
$$
S_{\wt R}^{(x,0)}(f,\om) = {\Phi_\fu^{\vec{s}}(x[0,N-1],\om)}.
$$
By construction, there exist words $u,v,w$ such that
$$
uv = \zeta^{[n]}(b),\ \ vw = x[0,N-1],\ \ \mbox{where}\ \ \max\{|u|_{\vec{s}},|w|_{\vec{s}}\} \le \tau,
$$
hence
$$
\max\{|u|,|v|\} \le \tau/s_{\min}.
$$
By (\ref{eq-Phi}),
\begin{eqnarray}
\bigl|S_{R_n}^{(x,t)}(f,\om)\bigr| & \ge & \bigl|\Phi_\fu^{\vec{s}}(x[0,N-1],\om)\bigr|  - 3{\|f\|}_\infty \cdot s_{\max} \nonumber \\[1.2ex]
& \ge & \bigl|\Phi_\fu^{\vec{s}}(\zeta^{[n]}(b),\om)\bigr| - \bigl|\Phi_\fu^{\vec{s}}(u,\om)\bigr| - \bigl|\Phi_\fu^{\vec{s}}(w,\om)\bigr|\label{dudu} -3{\|f\|}_\infty \cdot s_{\max}.
\end{eqnarray}
By  Lemma~\ref{lem-newup}, for $n\in \Nk$ sufficiently large:
\be \label{dudu2}
 \bigl|\Phi_\fu^{\vec{s}}(u,\om)\bigr| \le C_1 \cdot |u|^{(\chi^+_{\ba,\xi,\vec{z}}+\eps)/\lam_\ba}\le  \wt C_1 \cdot \tau^{(\chi^+_{\ba,\xi,\vec{z}}+\eps)/\lam_\ba},
\ee
and similarly for $\bigl|\Phi_\fu^{\vec{s}}(w,\om)\bigr|$. Recall that the constants depend on $\ba,\eps,\vec{s},\xi,\varphi$, but not on $N$, $n$, or $R$. By the hypothesis of positivity of the Lyapunov exponent, assume that  $\chi^+_{\ba,\xi,\vec{z}}>\eps>0$ without loss
of generality. Then by (\ref{dudu}), (\ref{dudu2}), and
 (\ref{byaka2}), for $n\in \Nk$ sufficiently large:
$$
\bigl|S_{R_n}^{(x,t)}(f,\om)\bigr|
\ge e^{(\chi^+_{\ba,\xi,\vec{z}}-\eps)n} - 2\wt C_1 \cdot \tau^{(\chi^+_{\ba,\xi,\vec{z}}+\eps)/\lam_\ba} -3{\|f\|}_\infty \cdot s_{\max}.
$$
We obtain for $n\in \Nk$ sufficiently large:
\be \label{byka}
\bigl|S_{R_n}^{(x,t)}(f,\om)\bigr| \ge (1/2) \cdot e^{(\chi^+_{\ba,\xi,\vec{z}} - \eps)n},
\ee
assuming
$$
\tau = c_2\cdot \exp\left(n\lam_\ba\cdot \frac{\chi^+_{\ba,\xi,\vec{z}} - \eps}{\chi^+_{\ba,\xi,\vec{z}} + \eps}\right),
$$
for an appropriate small constant $c_2>0$.
It follows from (\ref{byka}) and (\ref{byaka4}) that
\begin{eqnarray*}
\Jk:=\int_{\Xx_{\ba}^{\vec{s}}} |S_{R_n}^{(x,t)}(f,\om)|^2\,d\wt{\mu}(x,t) & \ge  & (1/4)\cdot e^{2(\chi^+_{\ba,\xi,\vec{z}}-\eps)n} \cdot \wt\mu(E_{n,\tau}) \\
& \ge & (1/4)\cdot e^{2(\chi^+_{\ba,\xi,\vec{z}}-\eps)n} \cdot \tau\cdot e^{-(\lam_\ba+\eps)n} \\[1.2ex]
& =  & \frac{c_2}{4}\cdot \exp\left[n\Bigl(2(\chi^+_{\ba,\xi,\vec{z}}-\eps) + \lam_\ba \cdot \frac{\chi^+_{\ba,\xi,\vec{z}} - \eps}{\chi^+_{\ba,\xi,\vec{z}}+\eps} - (\lam_\ba+\eps)\Bigr)\right]\\[1.2ex]
& = & \frac{c_2}{4}\cdot \exp\left[n(2\chi^+_{\ba,\xi,\vec{z}}- \eps')\right],
\end{eqnarray*}
for some $\eps'>0$, which tends to zero whenever $\eps$ tends to zero.
Recall that $$R_n = {|\zeta^{[n]}(b)|}_{\vec{s}} \le s_{\max}\cdot \|\Sf^{[n]}\|_1 \le e^{(\lam_\ba+\eps')n},$$
for large $n$, hence
$$
\Jk \ge \frac{c_2}{4}\cdot {R_n}^{(\gam-\eps')/(\lam_\ba+\eps')},\ \ \mbox{where}\ \gam = \frac{2\chi^+_{\ba,\xi,\vec{z}}}{\lam_\ba}\,,
$$
for $n$ sufficiently large, contradicting (\ref{byaka1}), if $\eps'>0$ is sufficiently small. The proof of the theorem is complete. 
\qed

\subsection{Proof of Theorem~\ref{th-main1110}} Let $x\in X_\ba$, and recall that by the recognizability assumption, we have the unique representation (\ref{recog}) for $n=\ell$, that is, $x = T^i \zeta^{[\ell]}(x')$ for some $x'\in X_{\sig^\ell\ba}$ and $0\le i < |\zeta^{[\ell]}(x_0')|$. For a Lip-cylindrical function of level $\ell$,
$$
f(x,t)=\sum_{j\in \Ak} \One_{\zeta^{[j]}[a]}(x) \cdot \psi^{(\ell)}_j(t),\ \ \mbox{with}\ \ \psi^{(\ell)}_j\in \Lip[0,s_j^{(\ell)}],
$$
holds
$$
f\circ h_\tau\bigl(\zeta^{[\ell]}(x'),0\bigr) = \psi_{x_n'}^{(\ell)} \bigl(\tau - {\bigl|x'[0,n-1]\bigr|}_{\vec{s}^{(\ell)}}\bigr),\ \ \mbox{for}\ {\bigl|x'[0,n-1]\bigr|}_{\vec{s}^{(\ell)}}\le \tau < {\bigl|x'[0,n]\bigr|}_{\vec{s}^{(\ell)}}.
$$
This implies
$$
S_{\wt R}^{(\zeta^{[\ell]}(x'),0)}(f,\om) = {\Phi_\fu^{\vec{s}^{(\ell)}}(x[0,N-1],\om)}\ \ \ \mbox{for}\ \ \wt R = \left|x'[0,N-1]\right|_{\vec{s}^{(\ell)}},\ \ \ \mbox{where}\ \ \fu(a) =  \widehat{\psi}^{(\ell)}_a(\om),
$$
which is the analogue of (\ref{SR1}). After that, the proof proceeds exactly as in Theorem~\ref{th-main10}. \qed

\begin{proof}[Proof of Theorems~\ref{th-main1} and \ref{th-main11}] %These follow from Theorems~\ref{th-main10} and \ref{th-main1110} respectively.
As written in the statement of Theorem~\ref{th-main1} and explained in the Introduction, for $\nu$-almost every Abelian differential $(M,\omb)$, there is a pair $(\ba,\vec s)$, with $\ba \in \Om_+$ and $\vec s\in \R^m_+$, such that the vertical translation flow on $(M,\omb)$ is measurably isomorphic to the suspension flow $(\Xx_\ba^{\vec{s}}, {\widetilde \mu}, h_t)$. This follows by a combination of \cite[Section 4.3]{Buf-umn} and Theorem~\ref{thm-recog}. To this end, we need to verify recognizability of the sequence of substitutions $\ba$ almost surely, with respect to the measure $\P_+$, the projection of the push-forward of $\nu$ by the symbolic coding to the positive coordinates. Theorem~\ref{th-recog0} provides sufficient conditions for recognizability, which hold in our case (almost surely). Indeed, the substitution matrices $\Sf_{\zeta_j}$, arising from the Rauzy-Veech induction, are unimodular \cite{veech}, hence have maximal rank. Aperiodicity of $X_\ba$ (almost surely) follows from the fact that admissible words in $\Om_+$ arise from walks in the Rauzy graph, which is aperiodic, so that the number of admissible words of length $n$ tends to infinity with $n$. Every admissible word appears in $\ba\in \Om_+$ with probability one.

Thus, we are position to apply Theorem~\ref{th-main10}, provided the conditions (A1$'$), (A2), and (A3) hold for $\P_+$-a.e.\ $\ba\in \Om_+$, and this is a consequence of Lemma~\ref{lem-vspoma1}. Note that (C1) and (C3) in the latter lemma follow by the fundamental results of Veech \cite{veech}, and condition (C2) holds by a theorem of Zorich \cite{Zorich1}. This concludes the derivation of Theorem \ref{th-main1}.

Theorem \ref{th-main11} follows from Theorem \ref{th-main1110}  in exactly the same way.
\end{proof}

%%%%%%%%%%%%%%%%%%%%%%%%%%%%%%%%%

\section{Proof of other results} \label{sec-add}

\begin{comment}
\begin{proof}[Proof of Lemma~\ref{lem-vspoma1}]

Property (A2) for $\P_+$-a.e.\ $\ba$ follows from ergodicity (C1) and integrability (C2) by the Furstenberg-Kesten theorem on the existence of Lyapunov exponent. The a.s.\ validity of (A3) is immediate from the Birkhoff-Khinchin Ergodic Theorem and (C2). Finally, the
property (A$1'$) for $\P_+$-a.e.\ $\ba$ is a consequence of (C1), (C3), and the following simple fact.

\medskip

\noindent {\bf Observation.} {\em
Let $(X,T,\mu)$ be an ergodic measure-preserving system and $A\subset X$ is measurable, with $\mu(A)>0$. Then for $\mu$-a.e.\ $x\in X$, for every $\eps>0$, there exists $n_0= n_0(x)$ such that 
$$
\forall\,n\ge n_0,\ \exists\, k \in [n(1-\eps),n]\cap \N:\ \ T^k x \in A.
$$
}
\noindent {\em Proof of the observation.}
By ergodicity, for a.e.\ $x$ we have $$n^{-1}\cdot \#\{k\in [0,n-1]\cap \N:\ T^k x\in A\}\to \mu(A),\ \ \mbox{ as}\ n\to \infty.$$
 Fix such a typical $x\in X$ and $\eps\in (0,1/2)$. Let $n_0=n_0(x)$ be such that 
$$
\left| n^{-1}\cdot \#\{k\in [0,n-1]\cap \N:\ T^k x\in A\} - \mu(A)\right| < \mu(A)\cdot\eps/2\ \ \mbox{for all}\ n\ge n_0/2.
$$
Then for all $n\ge n_0$,
$$
\#\{k\in [0,n(1-\eps)-1]\cap \N:\ T^k x\in A\} < \mu(A)\cdot n(1-\eps)(1+\eps/2),
$$
whereas
$$
\#\{k\in [0,n-1]\cap \N:\ T^k x\in A\} > \mu(A)\cdot n (1-\eps/2),
$$
whence
$$
\#\{k\in [n(1-\eps),n]\cap \N:\ T^k x\in A\} > \mu(A)\cdot n [(1-\eps/2) - (1-\eps)(1+\eps/2)] = \mu(A)\cdot \eps^2/2>0,
$$
as desired.
\end{proof}
\end{comment}

\begin{proof}[Proof of Corollary~\ref{cor-spec00}]
(i) The claim  follows from (\ref{ki2}) and  the following elementary fact:

\smallskip

{\em For any sequence of complex $m\times m$ matrices $\{A_n\}_{n\ge 1}$, there exists a coordinate basis vector $\vec{e}_j$ such that}
$$
\limsup_{n\to \infty} (n^{-1}\log\|A_n \vec{e}_j\|) = \limsup_{n\to \infty} (n^{-1}\log\|A_n\|). 
$$

(ii) Since $\Cc_\ba$ is continuous on $\T^m$, the pointwise upper Lyapunov exponents are measurable functions, and hence the sets
$$
\bigl\{(\om,\vec{b})\in \R\times \C^m:\ \chi^+_{\ba,\om\vec{s}, \Gam_\om \vec{b}}< \chi^+_{\ba,\om\vec{s}}\bigr\}
$$
are measurable for all $\vec{s}\in \R^m_+$. Part (i) implies that for a.e.\ $\om$, the set $\bigl\{\vec{b}\in \C^m:\ \chi^+_{\ba,\om\vec{s}, \Gam_\om \vec{b}}< \chi^+_{\ba,\om\vec{s}}\bigr\}$ has Lebesgue measure zero, and so the desired claim follows by an application of Fubini's Theorem and Theorem~\ref{th-main10}.
\end{proof}

We will need the following

\begin{lemma} \label{lem-approx} Let $\ba$ be a  primitive sequence of substitutions on $\Ak$ and $\vec{s}\in \R^m_+$. Then the collection of Lip-cylindrical functions of level $\ell$, for all $\ell\in \N$, is dense in $L^2(\Xx_\ba^{\vec{s}},\wt \mu)$.
\end{lemma}

\begin{proof} By the recognizability of $\ba$, the sequence 
$$
\{T^i(\zeta^{[\ell]}[a]):\ a\in \Ak,\ 0 \le i < |\zeta^{[\ell]}(a)|\},\ \ \ell\ge \ell_0,
$$
is a sequence of Kakutani-Rokhlin partitions that generates the Borel $\sig$-algebra on the substitution space, see \cite{BSTY}. It follows that the collection of functions of the form
\be \label{eq-approx}
f(x,t) = \sum_{a\in \Ak} \sum_{\ 0 \le i < |\zeta^{[\ell]}(a)|} \One_{T^i(\zeta^{[\ell]}[a])}(x)\cdot\psi^{(\ell)}_{i,a}(t),\ \ \mbox{where}\ \psi^{(\ell)}_{i,a}\in L^2\bigl[0,s_{\zeta^{[\ell]}(a)_i}\bigr],\ \ \ell\ge \ell_0,
\ee
is dense in $L^2(\Xxi^{\vec{s}},\wt \mu)$. Here $\zeta^{[\ell]}(a)_i$ is $i$-th letter of $\zeta^{[\ell]}(a)$.
The function $f$ in (\ref{eq-approx}) can be expressed as
$$
f(x,t) = \sum_{a\in \Ak} \One_{\zeta^{[\ell]}}(x) \cdot \psi_a^{(\ell)}(t), %\One_{[0,s^{(\ell)}_a]},
$$
where for all $a\in \Ak$,
$$
\psi^{(\ell)}_a(t) = \psi^{(\ell)}_{i,a}\bigl(t - \big|\zeta^{[\ell]}(a)[1,i-1]\bigr|_{\vec{s}}\bigr)
$$
whenever
$$
\big|\zeta^{[\ell]}(a)[1,i-1]\bigr|_{\vec{s}}\le t < \big|\zeta^{[\ell]}(a)[1,i]\bigr|_{\vec{s}},\ \ \ 0\le i < |\zeta^{[\ell]}(a)|-1.
$$
Here we are using the notation
$$
v[1,i] = v_1\ldots v_i,\ \ \mbox{for}\ \ v\in \Ak^n,\ n\ge i.
$$
Clearly, $\psi_a^{(\ell)}\in L^2\bigl[0,|\zeta^{[\ell]}(a)|_{\vec{s}}\bigr]$, $a\in \Ak$, hence
$f$ is an $L^2$-cylindrical function of level $\ell$. It remains to approximate
$\psi_a^{(\ell)}$, $a\in \Ak$, by Lipschitz functions to obtain a desired approximating Lip-cylindrical function of level $\ell$.
\end{proof}

\begin{corollary} \label{cor-approx.spec}
Let $\ba$ be a primitive sequence of substitutions on $\Ak$ and $\vec{s}\in \R^m_+$. 
If for every Lip-cylindrical function $f$ of level $\ell$, for all $\ell$ sufficiently large,
the spectral measure $\sig_f$ is singular with respect to the Lebesgue measure, then the flow $(\Xx_\ba^{\vec{s}},\wt \mu, h_t)$ has purely singular spectrum.
\end{corollary}
\begin{proof}
This is immediate from Lemma~\ref{lem-approx} and the following standard result (the proof for $\Z$-actions is in  \cite[Chapter 2]{Queff}; the proof for $\R$-actions is similar):
given a finite measure-preserving system, if $f_n \to f$ in $L^2$, then the spectral measures $\sig_{f_n}$ converge to $\sig_f$ in total variation.
\end{proof}

\begin{proof}[Proof of Corollary~\ref{cor-spec10}]
(i)
Suppose that $\chi^+_{\ba,\om\vec{s}} > \half\lam_\ba$ for $\om\in A\subset \R$, where $\Lk^1(A)>0$. By Corollary~\ref{cor-spec00}(ii), there exists a non-zero function $f\in L^2$ such that $\und{d}(\sig_f,\om) < 1$ for $\om\in A$. But then
$$
\limsup_{r\to 0} \frac{\sig_f(B_r(\om))}{2r} = +\infty,\ \ \om\in A,
$$
contradicting the well-known fact that $\lim_{r\to 0} \frac{\nu(B_r(\om))}{2r} < \infty$ Lebesgue-a.e.\ for any finite positive measure $\nu$, see e.g.\ \cite[Theorem 2.12(i)]{Mattila}.

To prove the second statement, note, see e.g.\ \cite[Prop.\ (10.12)]{Falc-book}, that $\sig_f$, being a measure on the line,  has upper Hausdorff dimension at most one, i.e. 
$$
\dim^*_H(\sig_f) = \inf \{s:\ \und{d}(\sig_f,\om) \le s\ \mbox{for} \ \sig_f\mbox{-a.e.}\ \om\} \le 1.
$$
It follows that $\und{d}(\sig_f,\om)\le 1$ for $\sig_f$-a.e.\ $\om$, and hence $\chi^+_{\ba,\om\vec{s}} \ge \half\lam_\ba$ for $\sig_f$-a.e.\ $\om$.

(ii)
Let $f\in L^2(\Xx_\ba^{\vec{s}})$ be Lip-cylindrical such that $\chi^+_{\ba,\om\vec{s}} < \half\lam_\ba$ for Lebesgue-a.e.\ $\om\in\R$. We want to show that  $\sig_f$ is singular.  Indeed, let $\sig_f = h\cdot \Lk^1 + (\sig_f)_{\rm sing}$, with $h\in L^1(\R)$, be the Lebesgue decomposition. 
Then by Theorem~\ref{th-main10}:
$$
\und{d}(\sig_f,\om) \ge 2 - \frac{2\max\{0,\chi^+_{\ba,\om\vec{s},\vec{z}}\}}{\lam_\ba} \ge 2 - \frac{2\max\{0,\chi^+_{\ba,\om\vec{s}}\}}{\lam_\ba} >1,
$$
for Lebesgue-a.e.\ $\om\in\R$.
But this implies
$$
\lim_{r\to 0} \frac{\sig_f(B_r(\om))}{2r} = 0
$$
for Lebesgue-a.e.\ $\om\in \R$, hence $h=0$ a.e.

(iii) Let $f$ be a Lip-cylindrical function of level $\ell$. By Theorem~\ref{th-main1110},
$$
\und{d}(\sig_{\!f},\om) \ge 2 - \frac{2\max\{0,\chi^+_{\ba,\om\vec{s}^{\,(\ell)}}\}}{\lam_\ba}\,.
$$
Under the assumption $\det\Cc_\ba(\om\vec{s},\ell)\not\equiv 0$, it follows that $\{\om\in \R:\, \det\Cc_\ba(\om\vec{s},\ell)=0\}$ is countable, since it is the set of zeros of a non-trivial trigonometric polynomial. Thus, we have for a.e.\ $\om$, by the definition of the cocycle,
$$
\Cc_\ba(\om\vec{s}^{(\ell)},n) = \Cc_\ba(\om\vec{s}, n+\ell)\cdot \Cc_\ba(\om\vec{s},\ell)^{-1},
$$
and hence
$\chi^+_{\ba,\om\vec{s}^{\,(\ell)}} =  \chi^+_{\ba,\om\vec{s}} < \lam_\ba/2$.
Then we deduce as in part (ii) that $\sig_f$ is purely singular, and Corollary~\ref{cor-approx.spec} implies the desired claim.
\end{proof}

%%%%%%%%%%%%%%%%%%%%%%%%%%%%%%%%%%%%%%%%%%%

\section{Appendix: on the H\"older property for the spectrum of translation flows in genus two} \label{section-stratum}

Here we return to translation flows on flat manifolds, see Section 1.1, and consider the vertical flow $h_t^+$.
Katok \cite{katok} proved that the flow $h_t^+$ is never mixing. 
%The moduli space of abelian differentials carries a natural volume measure, called the Masur-Veech measure. %\cite{masur}, \cite{veech}.  
%For almost every Abelian differential with respect to this measure,  Masur \cite{masur} and Veech \cite{veech}  independently proved that the flow   $h_t^+$ is
%uniquely ergodic.  
Under additional assumptions on the combinatorics of the abelian differentials, weak mixing for almost all  translation flows has been established by Veech in \cite{veechamj} and in full generality by Avila and Forni \cite{AF} (for genus $\ge 2$). Weak mixing means that spectral measures for test functions of mean zero are continuous, and it is a natural question (raised by Ya. Sinai in a personal communication) whether one can obtain quantitative estimates for the modulus of continuity for the spectral measures. The first result in this direction was obtained in \cite{BuSo2}; here we extend it to obtain the following.

\begin{theorem}\label{main-moduli}
There exists $\gamma>0$ such that  for almost every, with respect to the Masur-Veech measure, abelian differential $(M, \omb)$ on a surface of genus 2, the following holds. For any $B>1$ there exist constants $C=C(\omb,B)$ and $r_0=r_0(\omb,B)$ such that for any Lipschitz function $f$ on $M$, for all $x\in [B^{-1},B]$ and $r\in (0, r_0)$ we have 
$$
\sig_f([x-r, x+r])\le C\|f\|_L\cdot r^\gamma.
$$
\end{theorem} 

\begin{proof}
As is well-known, the moduli space of abelian differentials on a surface of genus 2 is a disjoint union of two strata: $\Hk(2)$, corresponding to $\omb$ with one zero of order two, and $\Hk(1,1)$, corresponding to $\omb$ having two simple zeros. Theorem~\ref{main-moduli} was proved in \cite{BuSo2} for the case of $\Hk(2)$, and here we indicate how a  modification of the proof yields the result for the stratum $\Hk(1,1)$.

The translation flow on the surface can be realized as a suspension flow over
an interval exchange transformation (IET), see Section~\ref{subsec-IET}. Let $\Hk$ be a stratum of  the moduli space of abelian differentials on a surface of genus $g \ge 2$.
The symbolic representation of a.e.\ translation flow as a suspension over a Bratteli-Vershik system was obtained in \cite{Buf-umn}, see Section~\ref{subsec-symbol} for more details.

Now let us restrict ourselves to the stratum $\Hk(1,1)$, which corresponds to the Rauzy class of the IET with permutation $(5,4,3,2,1)$. By the Avila-Viana Theorem \cite{AV}, the Rauzy-Veech cocycle in this case is known to have a simple Lyapunov spectrum $\theta_1 > \theta_2 > \theta_3=0 > \theta_4 = -\theta_2 > \theta_5 = -\theta_1$. However, according to Veech \cite{veechamj}, see also
\cite{AF}, the vector of ``heights'' for the suspension flow over the IET necessarily belongs to the subspace $H(\pi)$, which is invariant under the cocycle and depends only on the permutation. In fact, the space $H(\pi)$ is the sum of the stable and unstable subspaces for the cocycle (orthogonal to the central space with respect to the natural inner product). Thus the restriction of the cocycle to the sequence of invariant subspaces $H(\pi)$ has only non-zero Lyapunov exponents $\theta_1 > \theta_2 > -\theta_2 > -\theta_1$. 
%(This corresponds to passing from relative to absolute real cohomologies, see \cite{veech}. Note also that the translation surface is ``glued'' from the zippered rectangles of Veech accoding %to the ``glueing heights'' \cite{veech}, however, the vertical flow depends only on the heights of the rectangles and not on glueing heights.) 

It is important to note that the Masur-Veech measure on the stratum $H(1,1)$ is taken under this correspondence to a measure absolutely continuous with respect to the product of the measure $\P_+$ on $\Om_+$ (see Section~\ref{subsec-symbol} for the definition) and the Lebesgue measure on the set of possible height vectors in the subspace $H(\pi)$ defining the suspension. Thus the desired statement follows from a modified \cite[Theorem 4.1]{BuSo2} on suspension flows over random BV-transformations.

Everything in the statement of  \cite[Theorem 4.1]{BuSo2} remains unchanged, except that 
the assumption (a) on the Lyapunov spectrum: $\theta_1 > \theta_2 > 0 > \theta_3 > \ldots$ is replaced by

$(\mbox{a}')$ the Lyapunov spectrum satisfies
$$
\theta_1 > \theta_2 > 0 = \theta_3 > \ldots,
$$
and in the conclusion ``Lebesgue-a.e.\ vector $\vec{s}$,'' which determines the roof function, is replaced by ``Lebesgue-a.e.\ vector $\vec{s}$ in the subspace spanned by the stable and unstable Oseledets subspaces'' (thus excluding the zero exponent). 
The proof carries over verbatim. The only additional observation needed is that the equation \cite[(9.1)]{BuSo2} still holds, since the the decomposition of the vector $\vec{s}$ with respect to the Oseledets basis \cite{oseled} does not contain the term from the central subspace. We thus continue with the proof using the coordinates with respect to the two positive Lyapunov exponents to exclude the ``bad set,'' without any further changes.
\end{proof}

%%%%%%%%%%%%%%%%%%%%%%%%%%%%%%%%%%%%%%%%%%%%

\noindent {\bf Acknowledgements.}
We are deeply grateful to the anonymous referee whose comments helped us to improve the presentation.
The research of A. Bufetov has received funding from the European Research Council (ERC) under the European Union's Horizon 2020 
research and innovation programme under grant agreement No 647133 (ICHAOS).  A. Bufetov has also been funded by RFBR grant 18-31-20031.
The research of B. Solomyak was  supported by the Israel Science Foundation (grants 396/15 and 911/19).
We would like to thank the Institut Mittag-Leffler  of the Royal Swedish Academy of Sciences, where this work started, for its warm hospitality.  Part of this work was done while A.B. was visiting Bar-Ilan University and while B.S. was visiting CIRM Luminy in the framework of the ``research in pairs'' programme. We are  deeply grateful to these institutions for their hospitality.

\end{document}